\documentclass[12pt]{amsart}
\usepackage{graphicx}
\usepackage{graphicx}
\usepackage{epsfig}
\usepackage{pstricks}
\usepackage{eucal}
\usepackage{tikz}
\usetikzlibrary{arrows}
\usepackage{framed}

\newtheorem{thm}{Theorem}[section]

\newtheorem{defn}[thm]{Definition}
\newtheorem{lemma}[thm]{Lemma}

\newtheorem{cor}[thm]{Corollary}
\newtheorem{remark}[thm]{Remark}
\newtheorem{example}[thm]{Example}

\usepackage{amsmath}
\usepackage{amsxtra}
\usepackage{amscd}
\usepackage{amsthm}
\usepackage{amsfonts}
\usepackage{amssymb}
\usepackage{eucal}

\newcommand{\bmb}{\left( \begin{array}{rr}}
\newcommand{\enm}{\end{array}\right)}
\newcommand{\wQ}{\widehat Q}
\newcommand{\wbeta}{\widehat{\beta}}
\newcommand{\wv}{\widehat{v}}
\newcommand{\wrho}{\widehat{\rho}}
\newcommand{\of}{\overline{F}}
\newcommand{\oE}{\overline{E}}

\newcommand{\g}{{\mathfrak{g}}}
\newcommand{\h}{{\mathfrak{h}}}

\newcommand{\n}{{\mathfrak{n}}}

\newcommand{\half}{\frac12}

\newcommand{\bmu}{{\boldsymbol \mu}}
\newcommand{\bnu}{{\boldsymbol \nu}}

\newcommand{\C}{{\mathbb C}}
\newcommand{\Z}{{\mathbb Z}}
\newcommand{\Q}{{\mathbb Q}}
\newcommand{\R}{{\mathbb R}}
\newcommand{\N}{{\mathbb N}}

\newcommand{\bp}{{\mathbf p}}

\newcommand{\al}{{\alpha}}

\numberwithin{equation}{section}

\begin{document}

\title{A path model for Whittaker vectors}
\author{Philippe Di Francesco} 
\address{PDF: Department of Mathematics, University of Illinois MC-382, Urbana, IL 61821, U.S.A. e-mail: philippe@illinois.edu}
\author{Rinat Kedem}
\address{RK: Department of Mathematics, University of Illinois MC-382, Urbana, IL 61821, U.S.A. e-mail: rinat@illinois.edu}
\author{Bolor Turmunkh}
\address{BT: Department of Mathematics, University of Illinois MC-382, Urbana, IL 61821, U.S.A. e-mail: turmunk2@illinois.edu}
\date{\today}
\begin{abstract}
In this paper we construct weighted path models to compute Whittaker vectors in the completion of Verma modules, as well as Whittaker functions of fundamental type, for  all finite-dimensional simple Lie algebras, affine Lie algebras, and the quantum algebra $U_q(\mathfrak{sl}_{r+1})$. This leads to series expressions for the Whittaker functions. We show how this construction leads directly to the quantum Toda equations satisfied by these functions, and to the $q$-difference equations in the quantum case. We investigate the critical limit of affine Whittaker functions computed in this way.
\end{abstract}

\maketitle
\date{\today}

\section{Introduction}
Whittaker functions are fundamental objects in classical representation theory, which relate quantum integrable systems (the quantum Toda hierarchy and its generalizations) to Lie theory \cite{Kostant,KostantToda}. Whittaker functions have also been defined for affine, as well as quantum, algebras \cite{Etingof}, and satisfy corresponding q-difference equations. In addition to their representation theory connections,
these functions are important in the study of integrable systems, number theory, harmonic analysis, probability theory and algebraic geometry, and have been studied from these points of view by many authors (see for instance \cite{Lebedev,Lebedev3,Cherednik,Cherednik2,BUMP,Borodin,Oconnell,Givental,Lebedev2}). More recently,  analogues of Whittaker vectors 
for W-algebras called Gaiotto vectors have proved to be essential ingredients the AGT correspondence \cite{AGT}.

Eigenfunctions of the quantum Toda Hamiltonian are Whittaker functions corresponding to the classical Lie groups. They can be classified either according to their analytic properties or by choosing an appropriate class of representations through which they are obtained. The latter Whittaker functions can be defined using Whittaker vectors, which are elements in the (possibly completed) representation space.
The so-called {\it fundamental} Whittaker functions are obtained from irreducible Verma modules of the Lie algebra associated to the Lie group, whereas the so-called {\it class I} Whittaker functions are related to the principal series representations of the Lie group. This latter type of eigenfunctions can be obtained \cite{Hashizume} as linear combinations of fundamental solutions. 
In this paper, we concentrate on the construction of Whittaker vectors in Verma modules, 
and their corresponding fundamental Whittaker functions. So far, our methods do not apply to 
other types of representations such as principal series.

Recently, the notion of Whittaker vector has been generalized to the case of the 
$q$-deformed Virasoro or $W$-algebras,
where they are known as the Gaiotto vector \cite{AwataYamada}. The norm of this vector is the central object of the so-called 
AGT conjecture, relating
conformal blocks of Liouville field theory to the so-called Nekrasov partition functions,
defined as some instanton sums computing amplitudes of supersymmetric gauge theory. 

Although much is known on Whittaker functions, Whittaker vectors have proved surprisingly 
difficult to compute.
In this paper we present a general explicit construction of Whittaker vectors associated with Verma modules for simple and 
affine Lie algebras, as well as the quantum algebra $U_q(\mathfrak{sl}_{r+1})$.
We use a unified approach involving statistical sums over weighted paths. This can be considered a combinatorial representation-theory method, instead of an analytic one. (After writing this paper, we became aware of a similar idea contained in\cite{GoodWall} for the case of simple Lie algebras).

The method is as follows. By using the definition of the Whittaker vector in the irreducible highest weight Verma module, together with a choice of spanning set of the module indexed by directed paths on the positive root lattice, we find a choice of path weights, factorized into local contributions, such that the corresponding linear combination gives the Whittaker vector. The Whittaker function is then obtained as the matrix element of a generic Cartan subgroup element, between the Whittaker vector and its dual. It satisfies the appropriate Toda type differential or difference equations as a consequence of the recursion relations satisfied by the partition function of the path model, i.e. the weighted statistical sum over paths with a fixed end.

The Whittaker function of the  affine case at the critical level, when the central element acts by a scalar equal to minus the dual Coxeter number, is more subtle. In this paper, we show that the non-critical Whittaker function has an essential singularity in the critical limit, and admits a subleading asymptotic expansion, whose leading term is an eigenfunction of the critical affine Toda differential operator.
This case is of great interest, because the center of the universal enveloping algebra of the affine algebra becomes infinite-dimensional \cite{FrenResh} in the critical limit, and contains the deformed classical (Poisson) $W$-algebra.

Let us briefly define Whittaker vectors and functions in the classical case, and describe our results. Let $V$ be an irreducible, highest weight Verma module of the algebra $\g$, where $\g$ is either a simple or affine Lie algebra. Let $\n$ be the positive nilpotent subalgebra. We look for an element $w$ in the completion of $V$ on which $U(\n)$ acts by non-trivial characters. That is, if $\{e_i\}_{i\in I}$ are the Chevalley generators of $\n$, then $e_i w = \mu_i w$ with $\mu_i\in \C^*$ for all $i$. We can also say that $w$ is a simultaneous eigenvector of the elements of the nilpotent subalgebra. Then $w$ is called a Whittaker vector. It is an infinite sum and exists only in an infinite-dimensional module. It is uniquely determined up to a scalar multiple which is fixed by requiring the coefficient of the highest weight vector of $V$ in $w$ to be 1. A similar treatment using the opposite nilpotent subalgebra gives the Whittaker vector in the restricted dual of $V$. The matrix element of a generic element of the associated Cartan group between the Whittaker vector and its dual is the fundamental Whittaker function. It satisfies a quantum Toda differential equation, related to the action of the second Casimir element in the universal enveloping algebra.

Our method proceeds as follows.
Consider the spanning set $\{ f_{i_N}f_{i_{N-1}}\cdots f_{i_1} v, \, N\in \Z_+,\,  i_\ell\in I\}\subset V$, 
$\{f_i\}_{i\in I}$ the Chevalley generators of $\n_-$, $v$ the cyclic highest weight vector of $V$. 
Each such vector can be thought of as a path
$\bp$ on the positive cone of the root lattice $Q_+=\oplus_{i\in I} \Z_+\al_i$, starting at the origin, 
and taking successive elementary steps $\al_{i_1},\cdots \al_{i_N}$. We look for a Whittaker vector $w$ in the form of a linear combination
of such path vectors $w=\sum_\bp x(\bp) \bp$. The crucial idea behind the present work is that instead of looking for expressions for $w$ 
in a {\it basis} that incorporates the Serre relations of the algebra, we choose to write $w$ as a linear combination of {\it non independent}
vectors $\bp$. There is an infinite freedom in writing such expressions. Our main result 
(Theorem \ref{construc} for the finite case, and Theorem \ref{construcaff} for the affine case) is the construction of a particular, remarkably 
simple solution, in which the weight $x(\bp)$ factorizes into a product over the vertices visited by the path. The coefficient $x(\bp)$ can
therefore be thought of as a product over local Boltzmann weights for the path $\bp$.

Our approach also provides an alternative elementary derivation of the Toda type differential equations for the Whittaker functions
(Theorem \ref{todaW} for the finite case and Theorem \ref{todafW} for the affine case).
Indeed, using our particular solution, we derive easily an expression for the corresponding Whittaker functions, which admit a series expansion
with coefficients given by the path partition functions $Z_\beta=\sum_{{\rm paths}\, \bp: \,0\to \beta} x(\bp)$ for all $\beta\in Q_+$. 
The factorized form of $x(\bp)$ allows to write a simple recursion relation for these coefficients, which turns into the expected Toda type 
differential  equation for the Whittaker function. 

In this paper, we also generalize this construction to the quantum algebra $U_q(\mathfrak{sl}_{r+1})$  
(Theorems \ref{construcq} and \ref{todaqW}). 

The paper is organized as follows. We give the basic notation for the necessary data on simple and affine Lie algebras in Section 2. In Section 3, we compute the Whittaker vectors in the Verma module for finite-dimensional simple Lie algebra. We do the same for the affine algebras in Section 4. In Section 5, we compute the critical limit of the Whittaker function of the affine algebra. Section 6 is devoted to the generalization of our approach to the quantum algebra of type $A$. We gather a few concluding remarks in Section 7.
\medskip

\noindent{\bf Acknowledgments.} We thank O.Babelon, M.Bergvelt, A.Borodin, I.Corwin, N.Reshetikhin and S.Shakirov for discussions at various stages of this work.
R.K.'s research is supported by NSF grant DMS-1100929. P.D.F. and B.T. are supported by the NSF grant DMS-1301636 and the 
Morris and Gertrude Fine endowment. R.K. would like to thank the Institut de Physique Th\'eorique (IPhT) of Saclay, France, for hospitality during 
the completion of this work.

\section{Notations and Definitions}
Before presenting our results, we fix the combinatorial data defining the simple and affine Lie algebras. 

\subsection{Roots and weights}
Let $\g$ be a finite simple Lie algebra or affine Lie algebra,
with a Cartan matrix $C$ of rank $r$, 
and a Cartan decomposition $\g=\n_-\oplus \h\oplus \n_+$. The set $I$ denotes the index set of the simple roots, that is
$I=\{1,...,r\}$ in the finite case and $I=\{0,..,r\}$ in the affine case. The 
simple roots of $\g$ are the set $\Pi=\{\alpha_i:i\in I\}$. 
The root lattice $Q$ is the $\Z$-span of $\Pi$ and the positive root lattice is denoted by $Q_+=\bigoplus_{i\in I} \mathbb{Z}_+\al_i$. 

We use the notation 
$\{e_i, f_i, h_i: i\in I\}$ for the Chevalley generators of $\n_+, \n_-, \h$, respectively. Then $\g$ is generated by the Chevalley generators as a Lie algebra, with 
relations
\begin{eqnarray*}
[h_i, e_j] = C_{j,i} e_j, \quad [h_i,f_j] = -C_{j,i} f_j, \quad [e_i, f_j] = \delta_{i,j} h_j
\end{eqnarray*}
together with the Serre relations
$$
({\rm ad}~e_j)^{1-C_{ij}}( e_i) = ({\rm ad}~f_j)^{1-C_{ij}}(f_i)=0 .
$$

The positive coprime integers $d_i$ are defined by requiring the symmetry $C_{i,j} d_j = C_{j,i}d_i$. Then the co-roots of $\g$ are $\{\alpha_i^{\vee}: i\in I\}$ with $\alpha_i=d_i \alpha_i^\vee$, where $d_0=1$ in the affine case.

In the finite dimensional case, the highest root of $\g$ is denoted by $\theta$ and defines 
the positive integers (marks) $a_i$ as $\theta = \sum_{i=1}^r a_i \alpha_i$. For the affine algebra, the null root $\delta = \theta + \alpha_0$ with the mark $a_0=1$. The co-marks are defined by $d_i a_i^\vee = a_i$, and the dual coxeter number of $\g$ is $h^\vee=\sum_{i=0}^r a_i^\vee$. 

We define an inner product on the $\Q$-span of $\Pi$ by $(\alpha_i|\alpha_j^\vee)=C_{ij}$. 
In the affine case, this is a proper subspace of $\h^*$ since $C$ is singular.
We define the pairing between $\h^*$ and $\h$ by $h_j(\alpha_i)= (\alpha_i|\alpha_j^\vee)$, extended by linearity.

Let $\{\omega_i: i\in I \}$ denote the fundamental weights, defined by
$(\omega_i|\alpha_j^\vee)=\delta_{i,j}$. That is, 
\begin{eqnarray} \label{pairing}
h_i(\omega_j)=\delta_{i,j}.
\end{eqnarray}
Define $\rho\in \h^*$ to be the element such that $(\rho | \al_i^\vee)=1$ for all $i\in I$. In the finite case, $\rho=\sum_{i=1}^r \omega_i$. In the affine case,  define the affine weights $\Lambda_0=\omega_0, \Lambda_i=a_i^\vee \Lambda_0+\omega_i$ $(i>0)$, and $\wrho=\sum_{i\in I}{\Lambda_i}=h^\vee \Lambda_0 + \rho$ where $\rho$ is the finite element.

We will need the following bilinear forms:
\begin{eqnarray}\label{pairingstwo}
(\delta| \alpha_i)= (\delta|\delta) =0; \quad (\delta | \Lambda_0)=1;\\ \nonumber
(\omega_i|\alpha_j^\vee) = (\Lambda_i|\alpha_j^\vee)=\delta_{i,j};\quad (\Lambda_0|\Lambda_0)=0.
\end{eqnarray}

\subsection{Verma modules}
In this paper, we work with the irreducible Verma modules $V_\lambda$ generated by a cyclic vector $v$ with highest weight $\lambda\in \h^*$ such that $h_i(\lambda)\in \C^*\setminus \N$ for all $i\in I$:
$$
h v = h(\lambda) v, \quad \n_+ v = 0, \quad V_\lambda = U(\n_-) v.
$$
The affine Cartan algebra contains a one-dimensional center with generator $c=\sum_{i=0}^r  a_i^\vee h_i$. It acts on any irreducible module by a constant denoted by $k$. One can write the affine highest weight as $k \Lambda_0 + \lambda$ where $\lambda$ is a weight with respect to the finite part of the Cartan subalgebra.

The Verma module is infinite-dimensional, but it decomposes into a direct sum of finite-dimensional eigenspaces $V[\mu]\subset V$ with respect to the action of $\h$:
$$
V[\mu]=\{ w\in V_\lambda : h w = h(\mu) w\}.
$$
The subspace $V[\mu]$ is non-trivial if and only if $\mu = \lambda - \beta$ where $\beta\in Q_+$. We define the restricted dual $V_\lambda^*$
of $V_\lambda$ as the completed direct sum of the spaces $V[\mu]^*$, where the pairing is the natural Hermitian inner product on a complex vector space.

For the benefit of the mathematician reading this paper we define Dirac's very useful bra-ket notation, and will use it extensively throughout the paper. Assume $V$ is a complex vector space with a Hermitian inner product, which gives the corresponding natural isomorphism $*: V\to V^*$. Let $w^*\in V^*$ and $v \in V$, then $\langle w|$ is the notation for $w^*\in V^*$, $|v\rangle$ is the notation for $v\in V$, and $\langle w|v\rangle = w^*(v)$ by definition. Moreover,  $A$ acts on $V$ then it has a right action on $V^*$ as follows: $\langle w | A = \langle A^*w|$. Then $\langle w |A| v \rangle := (\langle w| A) |v\rangle = \langle w | (A | v \rangle )$. Here, $A^*$ is the Hermitian conjugate.

In the present context, we have $V=V_\lambda$ is the Verma module, $V^*$ its restricted dual, and we denote the cyclic, highest weight vector $v\in V_\lambda$ by $|\lambda\rangle$. We have $\langle \lambda|$ the corresponding vector in the dual to $V_\lambda$, such that $\langle \lambda|\lambda\rangle=1$ and
$$
\langle \lambda | h = h(\lambda) \langle \lambda |, \quad \langle \lambda | \n_-=0,\quad
V^*_\lambda = \langle \lambda | U(\n_+).
$$

In Section 6, we will also use the corresponding definitions for the quantum universal enveloping algebra $U_q(A_r)$. The combinatorial information of the roots and weights is the same, but we will use a slightly different definition of Whittaker vectors and functions than the one which follows. 

\subsection{Whittaker vectors and fundamental Whittaker functions}

For the classical Lie algebras, we have the following definition.

\begin{defn} \label{defnWhitt}
Let $w$ be an element in the completion of $V_\lambda$ such that
\begin{enumerate}
\item $\langle \lambda | w\rangle = 1;$
\item $e_i w = \mu_i w$ for all $i\in I$, where $\mu_i\in \C^*$.
\end{enumerate}
Such an element is called a Whittaker vector.
\end{defn}
We denote the unique solution to this set of equations by $|w\rangle^{\mu}_\lambda$ or simply $|w\rangle$ when there is no room for confusion about the parameters. It is the Whittaker vector corresponding to the nilpotent character $\mu=(\mu_i)_{i\in I}$. Clearly, $w$ must be an infinite sum in $V_\lambda$ and hence in the completion of the vector space, and any finite-dimension module does not contain a non-trivial Whittaker vector.

The fundamental Whittaker function is defined as follows. Let ${}^{\mu'}_\lambda\langle w |$ be the element in the restricted dual of $V_\lambda^*$, satisfying the corresponding conditions for a dual Whittaker vector. Namely, $\langle w | \lambda\rangle = 1$ and $\langle w | f_i = \mu_i' \langle w |$ with $\mu_i'\in \C^*$.
\begin{defn} \label{WhittFunc} 
The fundamental Whittaker function is the matrix element
$$
W_\lambda^{\mu\mu'}={}_{\lambda}^{\mu'}\langle w | {\rm e}^{\sum_{i\in I} \phi_i h_i} | w\rangle_{\lambda}^{\mu}
$$
\end{defn}
We consider $\phi_i$ as formal variables for the moment. 
Important considerations of analyticity and convergence are at the heart of Whittaker theory.

In the rest of the paper, we will compute series expressions for these functions and their quantizations.

\section{Whittaker functions for finite dimensional Lie algebras}\label{pathsec}
The Verma module $V_\lambda$ is spanned by monomials in the generators of $\n_-$, $\{f_i\}$ with $i\in [1,...,r]$, corresponding to the simple roots, acting on the highest weight vector,
$$
u=f_{i_N} \cdots f_{i_1}|\lambda\rangle,
$$
with weight $\lambda-\beta$ where $\beta=\al_{i_1}+\cdots+\al_{i_N}\in Q_+$.
The set of all such vectors is not linearly independent, due to the Serre relations.
However, this set is in bijection with the set of directed paths on $\Q_+\simeq \Z^r$, 
with steps only in the direction of the simple roots, as we explain below.

\subsection{Path models}
Let $\bp$ be a path on the positive cone of the root lattice 
$Q_+$ with steps from the set $\Pi$, written as $\bp=(p_0,p_1,...,p_N)$. Here, $N$ is the length of the path, $p_0$ is the origin $0\in Q_+$, $p_N=\beta\in Q_+$, and $p_i-p_{i-1}\in \Pi$.
\begin{defn}\label{pathdef}
The set ${\mathcal P}_\beta$ the set of all such directed paths from the origin to $\beta$ in $Q_+$. 
\end{defn}
The trivial path $(p_0)\in {\mathcal P}_0$ is denoted by $\mathbf 0$.

The correspondence between a path $\bp=(p_0,p_1,...,p_N)$, with
$p_k-p_{k-1}=\al_{i_k}$, and vectors in $V_\lambda$ is the following:
\begin{equation}\label{pvector}
\bp=(p_0,p_1,...,p_N)\mapsto |\bp\rangle = f_{i_N} \cdots f_{i_1}|\lambda\rangle.
\end{equation}

For example, the path $\bp=(0,\al_1,\al_1+\al_2,2\al_1+\al_2,3\al_1+\al_2,3\al_1+2\al_2)$ for $r=2$, $N=5$, and $\beta=3\al_1+2\al_2$ is illustrated below. The corresponding vector is $|\bp\rangle =f_2 f_1^2f_2 f_1 |\lambda\rangle$.

\begin{center}
\begin{tikzpicture}
    \draw[step=2, black!20] (0,0) grid (6,4);
    \draw[red, very thick, ->] (0,0) to (2,0); 
	\draw[red, very thick, ->] (2,0) to (2,2);
\draw[red, very thick, ->] (2,2) to (4,2);
\draw[red, very thick, ->] (4,2) to (6,2);
\draw[red, very thick, ->] (6,2) to (6,4);
\draw (0,0) node [below left] {(0,0)};
\draw(0,0) node {$\bullet$};
\draw (1,0) node [below] {$\al_1$};
\draw (2,1) node [right] {$\al_2$};
\draw (3,2) node [above] {$\al_1$};
\draw (5,2) node [above] {$\al_1$};
\draw (6, 3) node [right] {$\al_2$};
\draw (6,4) node [right] {$(3,2)$};
\draw (6,4) node {$\bullet$};
\end{tikzpicture}
\end{center}

A {\it path model} can be defined by a {\it vertex weight}\footnote{The reader is advised of the two unrelated usages in this paper of the word ``weight": One refers to the representation theoretical definition, and the other is used in the context of statistical models, also known as a Boltzmann weight.}
function 
$v:Q_+ \rightarrow \mathbb{C}$. The Boltzmann weight of a path is the product of vertex weights along the vertices visited by the path. That is, $\bp=(p_0,...,p_N)\in {\mathcal P}_\beta$ has a Boltzmann weight $x(\bp)$, where $x({\mathbf 0}):=1$ and otherwise
\begin{equation}\label{xweight}
x(\bp):=\prod_{k=1}^{N} \frac{1}{v(p_k)}.
\end{equation}
\begin{defn}\label{pathmodef}
The {\it partition function} $Z_\beta$ is the sum over all paths in $\mathcal P_\beta$ of the Boltzmann weights of the paths:
\begin{equation}
Z_\beta = \sum_{\bp\in {\mathcal P}_\beta} x(\bp)
\end{equation}
\end{defn}

Notice that $Z_0=x({\mathbf 0})=1$. 
The partition function $Z_\beta$ satisfies the following straightforward recursion relation:
\begin{lemma}
We have:
\begin{eqnarray} \label{recur}
v(\beta) \, Z_\beta= \sum_{i=1}^r Z_{\beta-\al_i}
\end{eqnarray}
With the initial condition $Z_0=1$, this determines $Z_\beta$ for all $\beta\in Q_+$. 
\end{lemma}

\begin{defn} Given $\beta \in Q_+$, define the ``vector partition function"
\begin{equation}\label{partvector}
|Z_\beta \rangle =\sum_{\bp\in {\mathcal P}_\beta} x(\bp)\, |\bp \rangle
\end{equation}
where $|\bp\rangle$ is defined in \eqref{pvector}.
\end{defn}

Our goal is to find vertex weights $v(\beta)$ so that the Whittaker function can be written as a sum over vector partition functions.

\subsection{Path model for Whittaker vectors}
We need the following preliminary Lemma. 
\begin{lemma}\label{weightlem}
The vertex weight function $v:Q_+\rightarrow \mathbb{C}$ given by
\begin{equation}
v(\beta)=(\lambda+\rho | \beta) - \frac{1}{2} (\beta | \beta)
\end{equation}
is the unique solution to the difference equations
\begin{equation}\label{diffv}
v(\beta+\al_i)-v(\beta)=(\lambda-\beta | \al_i) \qquad (i=1,2,...,r)
\end{equation}
subject to the initial condition $v(0)=0$. 
\end{lemma}
\begin{proof}
To show that $v(\beta)$ satisfies the recursion relation, we compute
\begin{eqnarray*}
v(\beta+\alpha_i) &=& (\lambda+\rho|\beta + \alpha_i) - \half(\beta+\alpha_i|\beta+\alpha_i) \\
&=& v(\beta) + (\lambda|\alpha_i) + (\rho-\half\alpha_i|\alpha_i) - (\beta|\alpha_i) \\
&=& v(\beta) + (\lambda-\beta|\alpha_i),
\end{eqnarray*}
where we used the fact that $(\rho|\alpha_i)=d_i$ and $\half(\alpha_i|\alpha_i)=d_i$.
\end{proof}

In terms of coordinates, writing $\beta=\sum_{i=1}^r \beta_i \al_i$ and $\lambda=\sum_{i=1}^r \lambda_i \omega_i$, the above weight reads:
\begin{equation*}\label{weight}
v(\beta)\equiv{\bar v}(\beta_1,\ldots,\beta_r)=\sum_{i=1}^r (1+\lambda_i)d_i\beta_i -\sum_{i=1}^r d_i \beta_i^2-\sum_{1\leq i<j\leq r} d_jC_{i,j}\beta_i\beta_j,
\end{equation*}
and it satisfies the difference equations:
\begin{equation*}
{\bar v}(\beta_1,\ldots,\beta_{i-1},\beta_i+1,\beta_{i+1},\ldots,\beta_r)-{\bar v}(\beta_1,\ldots,\beta_r)= d_i (\lambda_i  -\sum_{j=1}^r C_{j,i}\beta_j).
\end{equation*}

We will construct the Whittaker vector $|w\rangle$ in the completion of $V_\lambda$ by expressing it as a linear combination of elements in the spanning set \eqref{pvector}:
\begin{equation}\label{ansatza}
|w\rangle=\sum_{\beta_\in Q_+}  \bmu^{\beta} \sum_{\bp\in {\mathcal P}_\beta} y(\bp) |\bp \rangle
\end{equation}
for some coefficients $y(\bp)$ to be determined. Here we have used the multi-index notation
$\bmu^\beta=\prod_{i=1}^r (d_i\mu_i)^{\beta_i}$ for the vectors $\bmu=(d_1\mu_1,...,d_r\mu_r)$ and $\beta=\sum_{i=1}^r \beta_i\al_i$.
As the vectors $\{|\bp\rangle,\ \bp\in {\mathcal P}_\beta\}$ are not linearly independent,
we have some freedom in the choice of the coefficients, and we will choose them so that they have a factorized form. 

The defining conditions for the Whittaker vector, $e_i|w\rangle=\mu_i|w\rangle$ for all $i$ result in a 
recursion relation for the coefficients in \eqref{ansatza}.

\begin{figure}
 \includegraphics[width=12.cm]{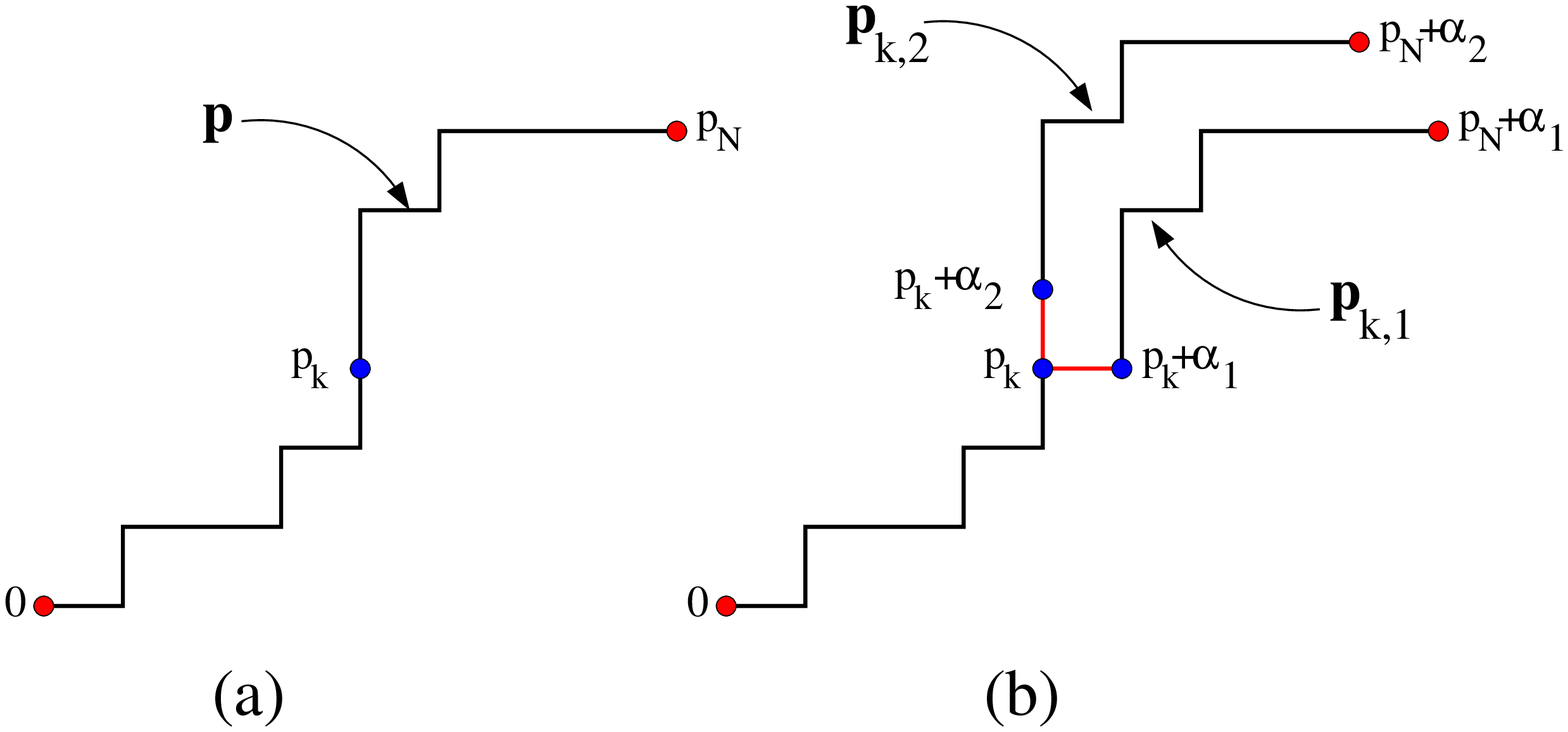}
\caption{\small A sample path $\bp$ for $A_2$ with a marked vertex $p_k$ (a) and the two corresponding augmented paths at $p_k$, 
$\bp_{k,1}$ and $\bp_{k,2}$ (b).}
\label{fig:augmented}
\end{figure}

\begin{lemma}
A sufficient condition for the vector $|w\rangle$ of \eqref{ansatza} to satisfy the Whittaker vector conditions of Definition \ref{defnWhitt} is that
the coefficients $y(\bp)$ obey the following linear system: 
\begin{equation} \label{eigencondition}
y(\bp)=\sum_{k=0}^{N} (\lambda- p_k | \al_i) y(\bp_{k,i}) \qquad {\rm for}\, {\rm all} \, \, \bp=(p_0,...,p_N)\in {\mathcal P}_\beta
\end{equation}
where the augmented path $\bp_{k,i}\in {\mathcal P}_{\beta+\al_i}$ is obtained 
by inserting a step 
$\al_i$ at the vertex $p_k$ of $\bp$ (see Fig.\ref{fig:augmented} for an illustration):
\begin{equation}\label{augmented} 
\bp_{k,i}=(p_0,p_1,...,p_k,p_k+\al_i,p_{k+1}+\al_i,...,p_N+\al_i) \end{equation}
\end{lemma}
\begin{proof}
Since $e_i|\bp\rangle=0$ if the path $\bp$ has no step $\al_i$, we may write
\begin{eqnarray*}
e_i|w\rangle&=& \sum_{\beta\in \al_i+Q_+}\bmu^{\beta-\al_i} d_i\mu_i 
\sum_{\bp\in {\mathcal P}_\beta\atop p=(p_0,...,p_{N+1})}
y(\bp)\sum_{k=0}^N \delta_{p_{k+1}-p_k,\al_i}\, \delta_{\bp'_{k,i},\bp}\, (\lambda-p_k|\al_i^\vee)\, 
|\bp'\rangle\\
&=& \mu_i \sum_{\beta'\in Q_+}\bmu^{\beta'}\sum_{\bp'\in {\mathcal P}_{\beta'}\atop \bp'=(p_0',...,p_N')} \sum_{k=0}^N 
y(\bp'_{k,i})\,  (\lambda-p_k|\al_i)\, |\bp'\rangle\\
\mu_i |w\rangle&=&\mu_i\sum_{\beta\in Q_+}\bmu^{\beta}\sum_{\bp\in {\mathcal P}_\beta} y(\bp)\, |\bp\rangle 
\end{eqnarray*}
The Lemma follows by identifying the coefficient of $|\bp\rangle$ in the last two lines. This is only a sufficient condition, as the 
vectors $|\bp\rangle$ are not linearly independent.
\end{proof}

Using these two Lemmas, we have the main result of this section (see also \cite{GoodWall}):
\begin{thm} \label{construc}
The unique vector in the completion of the Verma module of $\g$ with highest weight $\lambda$ which satisfies the properties of a Whittaker vector as in Definition \ref{defnWhitt} is
\begin{equation} \label{formula}
|w\rangle=\sum_{\beta\in Q_+}  \bmu^{\beta} |Z_\beta\rangle
\end{equation}
where $|Z_\beta\rangle$ are the vector partition functions \eqref{partvector} of the path
 model for $\g$ with the weights defined in Lemma \ref{weightlem}.
\end{thm}
\begin{proof}
The coefficient of $|\lambda\rangle$ in \eqref{formula} is 1, so condition (1) in Definition \ref{defnWhitt} is satisfied. We need to show that 
$e_i|w\rangle=\mu_i|w\rangle$ for each $i\in \{1,...,r\}$. 
That is, we must show that the coefficients $x(\bp)$ in Equation \eqref{xweight} satisfy 
\eqref{eigencondition}. 
By definition, $x(\bp) = \prod_{j=1}^N \frac{1}{v(p_j)}$ where $N$ is the length of $\bp$. Using \eqref{augmented}, $\bp$ has length $N+1$ and 
$$x(\bp_{k,j}) = \prod_{j=1}^k \frac{1}{v(p_j)} \prod_{j=k}^N \frac{1}{v(p_j+\alpha_i)}.$$ We have (since $v(p_0)=v(0)=1$)
\begin{eqnarray*}
\sum_{k=0}^{N}(\lambda - p_k | \al_i)x(\bp_{k,i})&=&
\sum_{k=0}^{N} \left(v(p_k+\al_i)-v(p_k)\right) \prod_{j=1}^{k}\frac{1}{v(p_j)}
\prod_{j=k}^{N} \frac{1}{v(p_j+\al_i)}\\
&=&\underbrace{-v(p_0)\prod_{j=0}^{N} \frac{1}{v(p_j+\al_i)} }_\text{$=0$} \\
&&\hspace{-1in}\quad +\sum_{k=0}^{N-1}
\prod_{j=1}^{k}\frac{1}{v(p_j)}\underbrace{\left( v(p_k+\al_i)\frac{1}{v(p_k+\al_i)} -\frac{1}{v(p_{k+1})}v(p_{k+1}) \right)}_\text{$=0$}  
\prod_{j=k+1}^{N} \frac{1}{v(p_j+\al_i)}\\
&&\hspace{-1in}\quad +\prod_{j=1}^{N}\frac{1}{v(p_j)} v(p_N+\al_i)\frac{1}{v(p_N+\al_i)}
=x(\bp)
\end{eqnarray*}
where we have rearranged the terms in the sum so that it is clear that everything but the last term in the summation over $k$ cancels. 
\end{proof}

\subsection{Rank 2 examples}

We give explicit examples for the case where $\g$ has rank 2, that is, 
$\g=A_2, B_2 \simeq C_2$, and $G_2$. 
The paths for all rank $2$ Lie algebras look the same, as $Q_+$ can be identified with $(\Z_+)^2$ in all cases.
\begin{center}
\begin{tikzpicture}
    \draw[step=2,black!20] (0,0) grid (8,6);
	\draw[very thick, ->] (0,0) to (9,0);
	\draw[very thick, ->] (0,0) to (0,7);

\draw(9,0) node [below] {$\al_1$};
\draw(0,7) node [left] {$\al_2$};

\draw (0,0) node [below left] {(0,0)};
\draw(0,0) node {$\bullet$};

\draw (2,0) node {$\bullet$};
\draw (2,0) node [below] {$(1,0)$};
\draw (4,0) node {$\bullet$};
\draw (4,0) node [below] {$(2,0)$};
\draw (6,0) node {$\bullet$};
\draw (6,0) node [below] {$(3,0)$};
\draw (0,2) node {$\bullet$};
\draw (0,2) node [left] {$(0,1)$};
\draw (0,4) node {$\bullet$};
\draw (0,4) node [left] {$(0,2)$};
\draw (2,2) node {$\bullet$};
\draw (2,2) node [above right] {$(1,1)$};
\draw (4,2) node {$\bullet$};
\draw (4,2) node [above right] {$(2,1)$};

\draw[red, very thick, ->] (0,0) to (2,0);
\draw[red, very thick, ->] (2,0) to (2,2);
\draw[red] (1,0) node [above] {$\al_1$};
\draw[red] (2,1) node [left] {$\al_2$};
\draw[blue, very thick, ->] (0,0) to (0,2);
\draw[blue, very thick, ->] (0,2) to (2,2);
\draw[blue] (0,1) node [left] {$\al_2$};
\draw[blue] (1,2) node [above] {$\al_1$};
\end{tikzpicture}
\end{center}

For each algebra, we consider two paths that start at the origin and end at the vertex $(1,1)$: $(0,\al_1,\al_1+\al_1)$ corresponding to the vector $f_2f_1|\lambda\rangle$ and $(0,\al_2,\al_1+\al_2)$ corresponding to the vector $f_1f_2|\lambda\rangle$. The difference between the Whittaker vectors is in the weights corresponding to each path.

\vskip.1in
\noindent{$A_2$:} In this case, $d_1=d_2=1$. Writing $\lambda=\lambda_1\omega_1+\lambda_2\omega_2$ and $\beta=\beta_1\al_1+\beta_2\al_2$, we have
\begin{eqnarray*}
v(\beta)={\bar v}(\beta_1,\beta_2)&=&(\lambda_1+1)\beta_1+(\lambda_2+1)\beta_2-\beta_1^2-\beta_2^2+\beta_1\beta_2\\
&=&
\beta_1(\lambda_1+1-\beta_1)+\beta_2(\lambda_2+1-\beta_2)+\beta_1\beta_2.
\end{eqnarray*}
For example, 
\begin{eqnarray*}
v(\al_1)&=&\bar{v}(1,0)=\lambda_1;\\
v(2\al_1)&=&\bar{v}(2,0)=2(\lambda_1-1);\\
v(\al_2)&=&\bar{v}(0,1)=\lambda_2;\\
v(\al_1+\al_2)&=&\bar{v}(1,1)=\lambda_1+\lambda_2+1;\\
v(2\al_1+\al_2)&=&\bar{v}(2,1)=2\lambda_1+\lambda_2.
\end{eqnarray*}
The Whittaker vector with character $\mu=(\mu_1, \mu_2)$ is given by:
\begin{eqnarray*}
|w\rangle_\lambda^\mu &=& 
{ \sum_{n\in \mathbb{Z}_+} \frac{\mu_1^n}{n! (\lambda_1)_n} f_1^n |\lambda\rangle}+{ \sum_{k\in \mathbb{Z}_+} \frac{\mu_2^k}{k! (\lambda_2)_k} f_2^k |\lambda\rangle}\\
&+&{\frac{\mu_1\mu_2}{\lambda_1(\lambda_1+\lambda_2+1)} f_2f_1 |\lambda\rangle 
+ \frac{\mu_1\mu_2}{\lambda_2(\lambda_1+\lambda_2+1)} f_1f_2 |\lambda\rangle} \\
&+&\frac{\mu_1^2\mu_2}{\lambda_1(\lambda_1+\lambda_2+1)(2\lambda_1+\lambda_2)} f_1f_2f_1 |\lambda\rangle 
+ \frac{\mu_1^2\mu_2}{\lambda_2(\lambda_1+\lambda_2+1)(2\lambda_1+\lambda_2)} f_1^2f_2 |\lambda\rangle\\
&+&\{\frac{\mu_1^2\mu_2}{(\lambda_1)2 (\lambda_1-1)(2\lambda_1+\lambda_2)}  f_2f_1^2|\lambda\rangle\ + ...
\end{eqnarray*}
\vskip.1in

\noindent{$B_2$:} Here, $d_1=2, d_2=1$. The vertex weight function is given by:
\begin{eqnarray*}
\bar{v}(\beta_1, \beta_2)&=&2(\lambda_1+1)\beta_1+(\lambda_2+1)\beta_2-2\beta_1^2-\beta_2^2+2\beta_1\beta_2\\
&=&2\beta_1(\lambda_1+1-\beta_1) + \beta_2(\lambda_2+1-\beta_2) + 2\beta_1\beta_2.
\end{eqnarray*}
The first few terms in the Whittaker vector are
\begin{eqnarray*}
|w\rangle_\lambda^\mu &=& \sum_{n\in \mathbb{Z}_+} \frac{(2\mu_1)^n}{2^n n! (\lambda_1)_n} f_1^n |\lambda\rangle + \sum_{k\in \mathbb{Z}_+} \frac{\mu_2^k}{k! (\lambda_2)_k} f_2^k|\lambda\rangle \\
&+&\frac{(2\mu_1)\mu_2}{(2\lambda_1) (2\lambda_1+\lambda_2+2)} f_2f_1|\lambda\rangle + \frac{(2\mu_1)\mu_2}{\lambda_2(2\lambda_1+\lambda_2+2)} f_1f_2|\lambda\rangle \\
&+& \frac{(2\mu_1)^2 \mu_2}{(2\lambda_1)(2\lambda_1+\lambda_2+2)(4\lambda_1+\lambda_2)} f_1f_2f_1|\lambda\rangle + \frac{(2\mu_1)^2\mu_2}{\lambda_2(2\lambda_1+\lambda_2+2)(4\lambda_1+\lambda_2)}f_1^2f_2|\lambda\rangle \\
&+&\frac{(2\mu_1)^2 \mu_2}{ 2^2(\lambda_1) 2(\lambda_1-1)(4\lambda_1+\lambda_2)} f_2f_1^2|\lambda\rangle +\cdots
\end{eqnarray*}

The case $C_2$ is obtained from $B_2$ be switching the indices $1$ and $2$.
In particular, we have $d_1=1, d_2=2$, and the vertex weight function is
\begin{eqnarray*}
\bar{v}(\beta_1,\beta_2)&=&(\lambda_1+1)\beta_1+2(\lambda_2+1)\beta_2-\beta_1^2-2\beta_2^2+2\beta_1\beta_2\\
&=&\beta_1(\lambda_1+1-\beta_1)+2\beta_2(\lambda_2+1-\beta_2)+2\beta_1\beta_2\ .
\end{eqnarray*}

\vskip.1in
\noindent{$G_2$:} Here, $d_1=1, d_2=3$ and vertex weight is
\begin{eqnarray*}
\bar{v}(\beta_1,\beta_2)&=&(\lambda_1+1)\beta_1+2(\lambda_2+1)\beta_2-\beta_1^2-3\beta_2^2+3\beta_1\beta_2\\
&=&\beta_1(\lambda_1+1-\beta_1)+3 \beta_2( \lambda_2+1-\beta_2) + 3\beta_1\beta_2.
\end{eqnarray*}

The first few terms in the Whittaker vector are
\begin{eqnarray*}
|w\rangle_\lambda^\mu &=& \sum_{n\in \mathbb{Z}_+} \frac{\mu_1^n}{n! (\lambda_1)_n} f_1^n |\lambda\rangle + \sum_{k\in \mathbb{Z}_+} \frac{(3\mu_2)^k}{3^k k! (\lambda_2)_k} f_2^k|\lambda\rangle \\
&+&\frac{\mu_1(3\mu_2)}{\lambda_1(\lambda_1+3\lambda_2+3)  } f_2f_1|\lambda\rangle + \frac{\mu_1(3\mu_2)}{(3\lambda_2)(\lambda_1+3\lambda_2+3)} f_1f_2|\lambda\rangle\\
&+&\frac{\mu_1^2(3\mu_2)}{\lambda_1 (\lambda_1+3\lambda_2+3)(2\lambda_1+3\lambda_2+4)}f_1f_2f_1|\lambda\rangle+ \frac{\mu_1^2(3\mu_2)}{ (3\lambda_2)(\lambda_1+3\lambda_2+3)(2\lambda_1+3\lambda_2+4)} f_1^2f_2|\lambda\rangle\\
&+& \frac{\mu_1^2(3\mu_2)}{(\lambda_1)2(\lambda_1-1)(2\lambda_1+3\lambda_2+4)}f_2f_1^2|\lambda\rangle +\cdots
\end{eqnarray*}

\subsection{The Whittaker function}

The Whittaker function is defined in Def. \ref{WhittFunc}. 
The dual Whittaker vector can be expressed in terms of path partition functions in the same way as the Whittaker vector itself. 
We associate a dual  vector $\langle \bp|\in V_\lambda^*$ to each path $\bp=(p_0,...,p_N)\in {\mathcal P}_\beta$ via:
$$\langle \bp|=\langle  \lambda| e_{i_1} e_{i_2}\cdots e_{i_k} $$ 
where $p_k-p_{k-1}=\al_{i_k}$ for $k=1,2,...,N$. 
In the same way as for the Whittaker vector, we have
$${}_\lambda^{\mu'}\langle w| =\sum_{\beta\in Q_+} (\bmu')^\beta \langle Z_\beta |$$
where $\bmu'=(d_1\mu_1',...,d_r\mu_r')$ and with the obvious definition for the left vector partition function for the path model
with the vertex weights $v$ of Lemma \ref{weightlem}, namely: $\langle Z_\beta |=\sum_{\bp\in{\mathcal P}_\beta}x(\bp)\langle \bp |$.

Before proceeding, let us introduce some more notations and definitions.

\begin{defn}
Define
\begin{eqnarray}
\phi &=& \sum_{i=1}^r \phi_i \al_i^\vee \label{phidef} \\
\nabla_\phi&=&\sum_{i=1}^r \frac{\partial}{\partial \phi_i} \omega_i\label{nabladef}
\end{eqnarray}
\end{defn}
Notice that we have
\begin{equation}
\frac{\partial}{\partial \phi_i} (\omega_i| \phi_j \al_j^\vee)= \frac{\partial \phi_j}{\partial \phi_i}(\omega_i, \al_j^\vee)=\delta_{ij}
\end{equation}

We have the following result: 

\begin{thm}\label{witpa}
The Whittaker function is the following generating series for the partition functions of the corresponding 
path model with fixed endpoints:
\begin{equation}\label{wifu}
W_\lambda^{\mu, \mu'}(\phi)=\sum_{\beta\in Q_+}  \bnu^\beta e^{(\lambda-\beta| \phi)} Z_\beta
\end{equation}
where $\bnu=(d_1\mu_1\mu_1',...,d_r\mu_r\mu_r')$. 
\end{thm}

In coordinates, the Whittaker function \eqref{wifu} can be written as:
\begin{equation}
W_{\lambda}^{\mu,\mu'} (\phi)=
\sum_{\beta\in Q_+} \prod_{i=1}^r (d_i\mu_i\mu'_i)^{\beta_i}e^{\sum_{i,j=1}^r 
\phi_i (\lambda_i\delta_{i,j}-C_{j,i}\beta_j)} Z_{\beta}
\end{equation}

\begin{proof}
Let us first compute the scalar product:
$\Sigma_\bp=\langle \bp | w\rangle$ between a left path vector with $\bp=(p_0,...,p_N)\in {\mathcal P}_\beta$ and the right Whittaker vector.
We get:
$$ \Sigma_\bp=\langle \lambda| e_{i_1}e_{i_2}\cdots e_{i_N} |w\rangle = 
\prod_{\ell=1}^N \mu_{i_{\ell}} \langle \lambda|w\rangle=\prod_{i=1}^r\mu_i^{\beta_i}$$
by use of the property $e_i|w\rangle=\mu_i|w\rangle$.
Using the result of Theorem \ref{construc}, we may now compute explicitly:
\begin{eqnarray*}
W_{\lambda}^{\mu,\mu'} (\phi)&=&{}_\lambda^{\mu'}\langle w| e^{\sum_i \phi_i h_i} |w\rangle_{\lambda}^\mu 
=\sum_{\beta\in Q_+} {\bmu'}^\beta  \langle Z_\beta| e^{\sum_i \phi_i h_i} |w\rangle_\lambda^{\mu} \\
&=&\sum_{\beta\in Q_+} {\bmu'}^{\beta} \, e^{(\lambda- \beta|\phi)}
\sum_{\bp\in {\mathcal P}_\beta}  x(\bp) 
\Sigma_\bp=\sum_{\beta\in Q_+}  \bnu^{\beta} e^{(\lambda-\beta| \phi)} Z_{\beta}
\end{eqnarray*}
where we have used the fact that for all $\bp\in \beta$ the left path vector $\langle \bp|$
has weight  $\lambda-\beta$, so that $\langle Z_\beta| h_i =(\lambda-\beta|\al_i^\vee)\langle Z_\beta|$.
\end{proof}

For later use, we introduce the modified Whittaker function defined as:
\begin{equation}
{\widetilde W}_{\lambda}^{\mu,\mu'} (\phi)= e^{\sum_{i=1}^r\phi_i}\, W_{\lambda}^{\mu,\mu'} (\phi) =\sum_{\beta\in Q_+} 
\bnu^\beta e^{(\lambda+\rho -\beta| \phi)} Z_\beta
\end{equation}
These modified Whittaker functions are usually referred to as {\it fundamental} Whittaker functions, and are characterized by their form as
power series expansions of the variables $e^{-\sum_j C_{i,j}\phi_j}$, which, up to the factor $e^{(\lambda+\rho|\phi)}$,
remain bounded when $\phi_i\to\infty$ in the domain $\sum_j C_{i,j}\phi_j\geq 0$ 
(i.e. $\phi\to\infty$ within the fundamental Weyl chamber of $\mathfrak{g}$). 

\begin{example}
For $G=A_1$, the modified Whittaker function ${\widetilde W}_{\lambda}^{\mu,\mu'} (\phi)$ is
$$ {\widetilde W}_{\lambda}^{\mu,\mu'} (\phi)=\sum_{\beta_1\in \Z_+} (\mu_1\mu_1')^{\beta_1} e^{\phi_1(\lambda_1+1-2\beta_1)}\,  Z_{\beta_1} $$
where the coefficient
$$Z_{\beta_1} =\frac{1}{(\beta_1)! \lambda_1 (\lambda_1-1)\cdots (\lambda_1+1-\beta_1)} $$
obeys the recursion relation ${\bar v}(\beta_1)Z_{\beta_1}=Z_{\beta_1-1}$, with ${\bar v}(\beta_1)=(\lambda_1+1)\beta_1-\beta_1^2$.
\end{example}

\begin{example}\label{sl3funex}
For $G=A_2$,  ${\widetilde W}_{\lambda}^{\mu,\mu'} (\phi)$ is
$$ {\widetilde W}_{\lambda}^{\mu,\mu'} (\phi)=\sum_{\beta_1,\beta_2\in \Z_+} 
(\mu_1\mu_1')^{\beta_1}(\mu_2\mu_2')^{\beta_2} e^{\phi_1(\lambda_1+1-2\beta_1)+\phi_2(\lambda_2+1-2\beta_2)}\,  Z_{\beta_1,\beta_2} $$
where the coefficient
$Z_{\beta_1,\beta_2}$
obeys the recursion relation 
\begin{equation}\label{eqa2} 
\left((\lambda_1+1)\beta_1+(\lambda_2+1)\beta_2-\beta_1^2-\beta_2^2+\beta_1\beta_2\right)
Z_{\beta_1,\beta_2}=Z_{\beta_1-1,\beta_2}+Z_{\beta_1,\beta_2-1}
\end{equation}
It is easy to show by induction that
\begin{equation}\label{bump} Z_{\beta_1,\beta_2}=\frac{
\prod_{j=1}^{\beta_1 + \beta_2} (\lambda_1 + \lambda_2 +2- j)}
{\prod_{j=1}^{\beta_1} j (\lambda_1+1 - j)(\lambda_1 + \lambda_2 +2- j) \prod_{j=1}^{\beta_2}
j (\lambda_2+1 - j)(\lambda_1 + \lambda_2 +2- j)}\end{equation}
This formula first appeared in \cite{BUMP}. Similar factorization properties in the case of $\mathfrak{sl}_3$ were observed in \cite{GoodWall}.
\end{example}

The main result about Whittaker functions is that they satisfy a Toda equation:
\begin{thm}\label{todaW}
The modified Whittaker function is a solution of the Toda differential equation associated with the root system of the Lie algebra $\g$:
\begin{equation}
H\, {\widetilde W}_{\lambda}^{\mu,\mu'} (\phi) = E_\lambda \, {\widetilde W}_{\lambda}^{\mu,\mu'} (\phi) 
\end{equation}
where the Hamiltonian operator $H$ and eigenvalue $E_\lambda$ are defined as:
\begin{eqnarray}
H&=&
\frac{1}{2} \left( \nabla_\phi | \nabla_\phi \right) + \sum_{j=1}^r \nu_j e^{- (\al_j|\phi)} \\
E_\lambda&=& \frac{1}{2} \left( \lambda+\rho | \lambda+\rho\right)
\end{eqnarray}
where $\nu_j=d_j\mu_j\mu_j'$.
\end{thm}

In coordinates, the Hamiltonian and the eigenvalue are given by:

\begin{eqnarray}
H&=&
\frac{1}{2}\sum_{i,j}d_j(C^{-1})_{i,j} \frac{\partial^2}{\partial \phi_i\partial\phi_j}
+\sum_{j=1}^r \nu_j  e^{-\sum_k C_{j,k}\phi_k } \label{defhg} \\
E_\lambda&=& \frac{1}{2}\sum_{i,j}d_j(C^{-1})_{i,j}(1+\lambda_i)(1+\lambda_j)\label{defeg}
\end{eqnarray}

\begin{proof}
Denote by $e(\phi)= e^{(\lambda+\rho -\beta|\phi)}$ and $K=E_\lambda-\frac{1}{2} \left( \nabla_\phi | \nabla_\phi \right) $. 
Using
$
(\nabla_\phi | \nabla_\phi)=\sum_{i,j} ( \omega_i |\omega_j) \frac{\partial^2}{\partial \phi_i \partial \phi_j}
$,
we may express the action of $K$ on $e(\phi)$ as:
\begin{eqnarray*} K\, e(\phi)& = &
\frac{1}{2}\left( (\lambda+\rho|\lambda+\rho)  
- \sum_{i,j} (\lambda+\rho - \beta| \al_i^\vee)(\omega_i|\omega_j)(\al_j^\vee|\lambda+\rho-\beta) \right)e(\phi) \\
&=&\frac{1}{2}\left( (\lambda+\rho| \lambda+\rho) - (\lambda+\rho-\beta| \lambda+\rho-\beta) \right) e(\phi)\\
&=&\left( (\lambda+\rho | \beta) - \frac{1}{2}(\beta|\beta) \right) e(\phi) = v(\beta) e(\phi) 
\end{eqnarray*}
Here, we used the fact that $\left\{\al_i^\vee\right\}$ 
and $\left\{\omega_j\right\}$ are dual bases. 

Let us now act with $K$ on the Whittaker function and use the recursion relation in Lemma \ref{recur} to rewrite:
\begin{eqnarray*}
K\, \widetilde{W}_\lambda^{\mu, \mu'}(\phi)&=& \sum_{\beta\in Q_+} \bnu^\beta e(\phi) v(\beta) Z_\beta
=\sum_{\beta\in Q_+} \bnu^\beta e^{(\lambda+\rho-\beta| \phi)} \sum_{j=1}^r Z_{\beta-\al_j}\\
&=&\sum_{j=1}^r \sum_{\beta\in Q_+} \bnu^{\beta+\al_j} e^{(\lambda+\rho -\beta|\phi) -(\al_j| \phi)} Z_\beta
=\sum_{j=1}^r \nu_j e^{-(\al_j, \phi)} \, W_\lambda^{\mu, \mu'}(\phi)
\end{eqnarray*}
The theorem follows.
\end{proof}

\begin{example}\label{todasl3}
In the case of $G=A_2$, the quantum Toda Hamiltonian \eqref{defhg} and its corresponding eigenvalue  \eqref{defeg} read:
\begin{eqnarray*} 
H&=& \frac{1}{3}\left(D_1^2 +D_1D_2+D_2^2\right)+\nu_1\, e^{-(2\phi_1-\phi_2)}+\nu_2\, e^{-(2\phi_2-\phi_1)} \\
E_{\lambda_1,\lambda_2}&=&\frac{1}{3}(\lambda_1^2+\lambda_1\lambda_2+\lambda_2^2)+\lambda_1+\lambda_2+1
\end{eqnarray*}
where we use the notation $D_i=\frac{\partial}{\partial \phi_i}$ for $i=1,2$.
\end{example}

\section{Path models for the Whittaker vector of affine Lie algebras: The non-critical level}\label{noncritsec}
In the case of the the irreducible Verma modules of non-twisted, affine Lie algebras, 
denoted by $\widehat{\g}$ in this section, the Whittaker vectors and fundamental Whittaker 
functions of the previous section can be extended in a straightforward manner, as will be 
shown in this section. There one additional feature in this case: 
The existence of the critical level, when 
 the central element of $\widehat\g$ acts by a scalar $k$ equal to 
$-h^\vee$, where $h^\vee$ is the dual coxeter number. In this case, our Boltzmann weights 
have singular values. This is a consequence of the fact that the affine algebra at the critical level 
has an infinite-dimensional center, isomorphic to a deformed classical $\mathcal W$-
algebra \cite{FrenResh} (as a Poisson algebra). This case is treated
separately in the next section, by using an asymptotic expansion of the formulas we 
derive in this section.

Let $V_\Lambda$ be an irreducible Verma module of  $\widehat \g$ with affine highest 
weight $\Lambda$. We use the notation $\Lambda = k \Lambda_0 + \lambda$, where $\lambda$ is a 
generic weight with respect to the finite Lie algebra, and $k$ is a complex number which is 
not a positive integer. We denote the cyclic, highest weight vector of $V_\Lambda$ by $|
\Lambda\rangle$.

The definition of the Whittaker vector in the completed Verma module is the same as in the finite-dimensional case:
The Whittaker vector in the irreducible Verma module $V_\Lambda$ with character $
\widehat\mu=(\mu_0,\mu_1,...,\mu_r)\in (\C^*)^{r+1}$
is the unique element $|w\rangle$ in the completion of $V_{\Lambda}$ such that 
$\langle 
\Lambda |w\rangle =1$ and:
\begin{equation}
e_i\, |w\rangle =\mu_i \, |w\rangle \qquad (i\in [0,r]).
\end{equation}

The path model defined in Section \ref{pathsec} 
for classical Lie algebra is generalized as follows. Paths have steps in 
$\Pi=\{\alpha_0,...,\alpha_r\}$, and are between the origin to an arbitrary point 
$\wbeta\in\wQ_+$. The set of all such paths is denoted by
${{\mathcal P}}_{\widehat \beta}$. The vertex weight function $\wv: 
\wQ_+ \to \C$ is
\begin{equation}\label{weightaff}
\wv(\widehat\beta)=( \Lambda+\wrho | \widehat \beta)- \frac{1}{2}(\widehat
\beta | \widehat\beta).
\end{equation}
\begin{lemma} The vertex weight function 
\eqref{weightaff} is the unique solution to the set of difference equations
\begin{equation*}
\wv(\widehat\beta+\al_i) - \wv(\widehat\beta)=(\Lambda-\widehat\beta | \al_i), 
\hspace{0.2in} (i\in [0, r])
\end{equation*}
subject to the initial condition $\wv(0)=0$. 
\end{lemma}
The proof is a straightforward generalization of Lemma \ref{weightlem}.

In terms of coordinates in $\wQ_+$, we can write
\begin{equation}
 \wv(\beta_0, \ldots, \beta_r)=\sum_{i=0}^r (\lambda_i+1)\beta_i d_i -\frac{1}
{2}\sum_{i,j=0}^r d_j\widehat C_{i,j}\beta_i\beta_j,
\end{equation}
and the difference equations are
\begin{equation}
 \wv(\beta_0,\ldots, \beta_i+1, \ldots, \beta_r)-\bar \wv(\beta_0, \ldots, 
\beta_r)
= d_i (\lambda_i  -\sum_{j=0}^r \widehat C_{j,i}\beta_j)\qquad (i\in [0,r]).
\end{equation}

The partition function $Z_{\widehat\beta}$ is defined in the same way as for the finite dimensional case. Clearly, it satisfies the recursion relation
\begin{equation}\label{recaff}
\wv(\widehat\beta)\, Z_{\widehat\beta}= \sum_{i=0}^r Z_{\widehat\beta-\al_i}.
\end{equation}

\begin{remark}\label{vaffrem}
It is easy to see that
\begin{equation}
\wv(\widehat\beta  + \widehat\gamma)=\wv(\widehat\beta)+\wv
(\widehat\gamma)-(\widehat \beta|\widehat\gamma)
\end{equation}
for any affine roots $\widehat \beta$ and $\widehat\gamma$.
Applying this to the root $\wbeta$ written as
\begin{equation*}
\widehat \beta=\sum_{i=0}^r \beta_i \al_i=\beta_0(\delta-\theta)+\sum_{i=1}^r\beta_i \al_i
=\beta_0\delta+(\beta-\beta_0\theta ),
\end{equation*}
where $\delta$ is the null root and $\theta$ the highest root of the finite dimensional algebra $\g$, we find
\begin{eqnarray*}
\wv (\widehat\beta)&=& \wv(\beta_0 \delta) +\wv(\beta - \beta_0 
\theta)-(\beta_0\delta | \beta -\beta_0\theta)\\
&=&( \Lambda+\wrho | \beta_0 \delta) +v (\beta - \beta_0 \theta)\\
&=& (k+h^\vee)\beta_0 +v (\beta - \beta_0 \delta).
\end{eqnarray*}
We have used the bilinear forms
$( \Lambda | \delta) = k$ and $(\wrho | \delta)= h^\vee$, the dual coxeter number,  from 
\eqref{pairingstwo}. Notice that 
$\widehat\beta-\beta_0\delta=\beta-\beta_0\theta$ is a non-affine weight and the affine 
vertex weight function $\wv$ on $\widehat\beta-\beta_0\delta$ reduces to the finite 
vertex weight function $v$. 
Unlike in the finite case, it is now possible to have $\wv(\widehat\beta)=0$ even if $
\widehat\beta\neq 0$. 
This situation happens exactly in the {\it critical} case, when the level takes the critical 
value $k=-h^\vee$, and when 
$\widehat\beta-\beta_0\delta=0$, i.e. $\beta_i=a_i \beta_0$ for $i\in [1,r]$.  

\end{remark}

We can now write the Whittaker vector:
\begin{thm}\label{construcaff}
For non-critical level $k\neq -h^{\rm v}$, the Whittaker vector for $\widehat \g$ in the 
completion of $V_{\Lambda}$ with
character $\widehat\mu=(\mu_0, \ldots, \mu_r)$ is given by
\begin{equation}
|w\rangle_{\Lambda}^{\widehat\mu}= \sum_{\widehat\beta\in \widehat Q_+}{\widehat \bmu}
^{\widehat\beta} \, |Z_{\widehat\beta}\rangle 
\end{equation}
where ${\widehat\bmu}=(d_0\mu_0, d_1 \mu_1, \ldots, d_r\mu_r)$, ${\widehat\bmu}
^{\widehat\beta}=\prod_{i=0}^r (d_i\mu_i)^{\beta_i}$ with $d_0=1$, and 
\begin{equation*}
|Z_{\widehat\beta}\rangle = \sum_{\bp\in {{\mathcal P}}_{\widehat \beta}} x(\bp) |\bp\rangle,
\end{equation*}
where the sum extends over paths on $\widehat Q_+$ that end at $\widehat \beta$,
with weight $x(\bp)$ (as in \eqref{xweight}) equal to the product of inverses of the vertex 
weights $\wv$ of eq.\eqref{weightaff},
at the non-zero vertices visited by the path $\bp$.
\end{thm}
\begin{proof}
The proof is identical to that of Theorem \ref{construc}.
\end{proof}

\subsection{Whittaker function}

We use the notation
$\widehat \phi=\sum_{i=0}^r \phi_i \al_i^\vee $, while keeping the notation
$\phi=\sum_{i=1}^r \phi_i \al_i^\vee$ for the finite part, and $\nabla_\phi$ for the gradient defined in
\eqref{nabladef}.

As before, the Whittaker function is defined for non-critical level as
\begin{equation}\label{whittaff}
W_{\Lambda}^{\widehat\mu,\widehat\mu'} (\widehat\phi)={}_\Lambda^{\widehat\mu'}\langle 
w| e^{\sum_{i=0}^r\phi_ih_i} |w\rangle_{\Lambda}^{\widehat\mu}.
\end{equation}

\begin{thm}\label{witpaff}
The Whittaker function at non-critical level $k\neq -h^{\vee}$ is the following generating 
series for the partition functions of the 
affine path model:
\begin{equation}
W_\Lambda^{\widehat\mu,\widehat \mu'}(\widehat\phi)=\sum_{\widehat \beta\in 
\widehat{Q}_+} {\widehat\bnu}^{\widehat\beta} 
e^{(\Lambda-\widehat\beta| \widehat\phi)} Z_{\widehat\beta}
\end{equation}
where 
${\widehat \bnu}^{\widehat\beta}=\prod_{i=0}^r (d_i\mu_i\mu_i')^{\beta_i}$. 
\end{thm}
\begin{proof}
The proof is identical to that of Theorem \ref{witpa}.
\end{proof}
As before, we define the 
modified Whittaker function as:
\begin{equation}\label{modiaff}
{\widetilde W}_{\Lambda}^{\widehat\mu,\widehat\mu'} (\widehat\phi)= e^{\sum_{i=0}^r\phi_i}
\, W_{\Lambda}^{\widehat\mu,\widehat\mu'} (\widehat\phi) 
=\sum_{\widehat\beta\in \widehat Q_+} {\widehat\bnu}^{\widehat\beta} e^{(\Lambda+
\wrho -\widehat\beta| \widehat\phi)}Z_{\widehat\beta}.
\end{equation}

Like in the simple case, we have the following differential equation for the non-critical Whittaker function.

\begin{thm}\label{todafW}
The modified Whittaker function at non-critical level $k\neq -h^{\rm v}$ is a solution of a 
deformed $\widehat G$-type Toda differential equation:
\begin{equation}\label{deftod}
\widehat H\, {\widetilde W}_{\Lambda}^{\widehat\mu,\widehat\mu'} (\widehat\phi) =  
E_\lambda \, {\widetilde W}_{\Lambda}^{\widehat\mu,\widehat\mu'} (\widehat\phi) ,
\end{equation}
where the affine Hamiltonian operator $\widehat H$ is
\begin{equation}\label{hamilto}
\widehat{H}=-(k+h^\vee) \nu_0 \frac{\partial}{\partial \nu_0} + \frac{1}{2} \left( \nabla_\phi | 
\nabla_\phi\right) 
+\sum_{j=0}^r \nu_j e^{-( \al_j| \phi)},
\end{equation}
and its eigenvaue is the same as in the finite case:
\begin{equation}
E_\lambda=\frac{1}{2}(\lambda+\rho|\lambda+\rho).
\end{equation}
Here, $\nu_j=d_j\mu_j\mu_j'$ for $j\in [0,r]$.
\end{thm}

\begin{proof}
Let
$e(\widehat \phi)=  {\widehat\bnu}^{\widehat\beta} e^{(\Lambda+\rho_{\rm 
aff} - \widehat\beta| \widehat \phi)}$, and
$\widehat K= E_\lambda+(k+h^{\rm v})\nu_0\frac{\partial}{\partial \nu_0}
-\frac{1}{2}(\nabla_\phi| \nabla_\phi)$. Then
\begin{eqnarray*}
\widehat{K} \,e(\widehat\phi) &=& \left( \frac{1}{2} (\lambda+\rho | \lambda+
\rho) + (k+h^\vee)\beta_0 \right. \\
&&\left. -\frac{1}{2} \sum_{i,j=1}^r (\omega_i|\omega_j)(\Lambda+\wrho - \widehat
\beta| \al_i^\vee) ( \Lambda+\wrho - \widehat\beta | \al_j^\vee)\right)e(\widehat\phi)\\
&=&\left(\frac{1}{2}(\lambda+\rho | \lambda+\rho) + (k+h^\vee) \beta_0 -\frac{1}{2}(\lambda+
\rho - \widehat\beta | \lambda + \rho -\widehat\beta)\right)e(\widehat\phi)\\
&=&\left((k+h^\vee)\beta_0 +(\lambda+\rho | \beta - \beta_0\theta) -\frac{1}{2}(\beta-
\beta_0\theta|\beta-\beta_0\theta)\right)e(\widehat\phi)\\
&=&\left((k+h^\vee)\beta_0+v(\beta-\beta_0\theta)\right)e(\widehat\phi, \widehat
\beta)=\wv (\widehat \beta)e(\widehat\phi),
\end{eqnarray*}
where we substituted $\widehat \beta=\beta_0\delta+(\beta-\beta_0\theta)$, and finally 
used Remark \ref{vaffrem}.

Acting on the Whittaker function and using the recursion relation \eqref{recaff},
\begin{eqnarray*}
\widehat K \, \widetilde{W}_\Lambda^{\widehat\mu,\widehat \mu'}(\widehat \phi)&=& 
\widehat K\, \sum_{\widehat \beta\in \widehat Q_+} e(\widehat\phi) 
Z_{\widehat \beta}
=\sum_{\widehat \beta\in \widehat{Q}_+} e(\widehat\phi) \wv(\widehat
\beta) Z_{\widehat\beta}\\
&=&\sum_{\widehat\beta\in \widehat{Q}_+} {\widehat\bnu}^{\widehat{\beta}} e^{(\Lambda+
\wrho-\widehat\beta|\widehat\phi)} \sum_{j=0}^r Z_{\widehat\beta - \al_j}
=\sum_{j=0}^r \sum_{\widehat\beta\in \widehat Q_+} {\widehat\bnu}^{\widehat\beta+\al_j} 
e^{(\Lambda+\wrho -\widehat\beta|\widehat\phi) - (\al_j|\widehat \phi)} 
Z_{\widehat{\beta}}\\
&=&\sum_{j=0}^r \nu_j e^{-(\al_j|\widehat \phi)} \widetilde{W}_\Lambda^{\widehat\mu, 
\widehat\mu'}(\widehat \phi).
\end{eqnarray*}
The theorem follows.
\end{proof}

\begin{remark}
We note that our Hamiltonian  $\widehat H$ of \eqref{hamilto} reduces to the critical affine Toda Hamiltonian computed by
Etingof \cite{Etingof}, when we set $k=-h^\vee$. 
We may therefore think of $\hat H$ as a deformation of the critical affine quantum Toda Hamiltonian.
\end{remark}

\begin{example}
In the case $\widehat G=\widehat A_r$, where $a_i^\vee=d_i=1$ for all $i$, 
The vertex weight function is given by
\begin{equation}\label{aaff}
v(\widehat\beta)=(k+r+1)\beta_0 +\sum_{i=1}^r (1+\lambda_i)
(\beta_i-\beta_0)-\frac{1}{2}\sum_{i,j=1}^r C_{i,j}(\beta_i-\beta_0)(\beta_j-\beta_0)
\end{equation}
and we get the following partial differential equation 
for the modified Whittaker function ${\widetilde W}(\widehat\phi):= {\widetilde W}
_{\Lambda}^{\widehat\mu,\widehat\mu'} (\widehat\phi)$, with $\nu_i=\mu_i\mu_i'$:
\begin{eqnarray}
&&\left\{(k+r+1)\nu_0\frac{\partial}{\partial \nu_0}+
\sum_{i,j=1}^r(C^{-1})_{i,j}\Big( \frac{(1+\lambda_i)(1+\lambda_j)}{2}-\frac{1}
{2}\frac{\partial^2}{\partial \phi_i\partial\phi_j}\Big)\right\}
{\widetilde W}(\widehat\phi)\nonumber \\
&&\qquad \quad \qquad\qquad \qquad\qquad\qquad \qquad \qquad\qquad \qquad \qquad 
= \left(\sum_{j=0}^r \nu_j\,e^{-\sum_{k=0}^r \widehat C_{j,k}\phi_k } \right) {\widetilde W}
(\widehat\phi).\label{todaff}
\end{eqnarray}

\end{example}

\section{The critical limit of the affine Whittaker function}
When the central element $c$ of the affine algebra $\widehat{\g}$ acts as a scalar multiple equal to $k=-h^\vee$, where $h^\vee$ is the dual coxeter number, the partition function  presented in the previous section is singular and needs to be treated separately.  
This case is called the critical limit. The model has a richer, integrable structure and the universal enveloping algebra of the affine algebra has an infinite-dimensional center, isomorphic to the classical (Poisson) deformed $\mathcal W$-algebras \cite{FrenResh}. In this limit, the Hamiltonian of the previous section becomes the affine Toda Hamiltonian \cite{Etingof}. The goal of this section is to study the behavior of the Whittaker function at the critical level. 

\subsection{The Whittaker function at critical level $k=-h^\vee$}

Let $\epsilon=k+h^\vee$, and consider the Whittaker function of Section \ref{noncritsec} when $\epsilon\to 0$.
Recall the expression for the Whittaker function ${\widetilde W}_\Lambda^{\widehat\mu,\widehat\mu'}(\widehat{\phi})$ of Theorem \ref{witpaff}:
$${\widetilde W}_\Lambda^{\widehat\mu,\widehat\mu'}(\widehat{\phi})=\sum_{\widehat \beta \in \widehat Q_+} Z_{\widehat \beta}\,{\widehat\bnu}^{\widehat\beta} 
e^{(\Lambda+\wrho - \widehat\beta | \widehat{\phi})}. $$
We separate out the terms proportional to $\epsilon$ in the exponential:
\begin{lemma}
\begin{equation}(\Lambda+\wrho-\widehat \beta | \widehat{\phi})=
\epsilon\phi_0+\left( \lambda+\rho-(\beta-\beta_0\theta)\vert(\phi-\phi_0 \theta)\right) .\label{identutile}
\end{equation}
\end{lemma}
\begin{proof}
Using
$ \Lambda+\wrho=(k+h^\vee)\Lambda_0 +\lambda+\rho$, 
$\widehat \beta=\beta_0\delta+(\beta-\beta_0\theta)$ and $\widehat{\phi}=\phi_0 \delta+(\phi-\phi_0 \theta)$, we have
\begin{eqnarray*}(\Lambda+\wrho-\widehat \beta | \widehat{\phi})&=&
\left( (k+h^\vee)\Lambda_0-\beta_0\delta+\lambda+\rho-(\beta-\beta_0\theta) \vert \phi_0 \delta+(\phi-\phi_0 \theta)\right)\\
&=& (k+h^\vee)\phi_0 (\Lambda_0|\delta) -\beta_0\phi_0 ( \delta |\delta) 
+\left(\lambda+\rho-(\beta-\beta_0\theta)|(\phi-\phi_0 \theta)\right).
\end{eqnarray*}
The Lemma follows using the relations \eqref{pairingstwo}.
\end{proof}

This motivates the definition of a renormalized Whittaker function:
\begin{defn}\label{renormW}
The renormalized Whittaker function is defined as
\begin{equation}\label{renorW}
{\widehat W}_\Lambda^{\widehat\mu,\widehat\mu'}(\widehat{\phi}):=e^{-\epsilon\phi_0}{\widetilde W}_\Lambda^{\widehat\mu,\widehat\mu'}(\widehat{\phi})=
\sum_{\widehat \beta \in \widehat Q_+} Z_{\widehat \beta}\,{\widehat \bnu}^{\widehat\beta} 
e^{\big(\lambda+\rho -(\beta-\beta_0\theta) \big| \phi- \phi_0 \theta\big)} 
\end{equation}
\end{defn}
This function depends on the variable $\widehat{\phi}$ only through the combination $\phi- \phi_0 \theta=\sum_{i=1}^r
(\phi_i-a_i^\vee \phi_0)\al_i^\vee$ and therefore satisfies the same eigenvalue equation \eqref{deftod} as 
${\widetilde W}_\Lambda^{\widehat\mu,\widehat\mu'}(\widehat{\phi})$. It is the unique solution (up to an overall factor independent of $\phi$)
with a series expansion of the form \eqref{renorW}.
This is convenient for computing the asymptotic expansion of the renormalized Whittaker function. It has an essential singularity at $\epsilon=0$ which is independent of the variable $\phi$, times a power series in $\epsilon$ whose terms depend on $\widehat\phi$. To get the coefficients, it is simplest to use the variables $\boldsymbol \nu$ as the expansion variables.

\begin{thm}
In the critical limit $\epsilon\to 0$, we have the following asymptotic expansion for the renormalized Whittaker function 
${\widehat W}_\Lambda^{\widehat\mu,\widehat\mu'}(\widehat{\phi})$ of Def.\ref{renormW}:
\begin{equation}\label{ansatz}
{\widehat W}_\Lambda^{\widehat\mu,\widehat\mu'}(\widehat{\phi})=e^{\frac{F}{\epsilon}}\left( W_0(\widehat\phi)+\epsilon W_1(\widehat\phi)+\cdots \right),
\end{equation}
where $F$ is a power series in the variable ${\widehat \bnu}^\delta=\prod_{i=0}^r(d_i\mu_i\mu_i')^{a_i^\vee}$ 
independent of $\widehat{\phi}$ of the form:
\begin{equation} \label{formF}
F=\sum_{\beta_0\in \Z_{>0}} {\widehat\bnu}^{\beta_0\delta} a_{\beta_0} 
\end{equation}
and the coefficients $\{W_j(\widehat\phi),\  j\geq 0\}$
are power series of the form: 
\begin{equation}\label{formWj}
W_j(\widehat\phi)=\sum_{\widehat\beta\in \widehat Q_+} 
{\widehat \bnu}^{\widehat \beta} \,e^{\big(\lambda+\rho-(\beta-\beta_0\theta)\big|(\phi-\phi_0 \theta)\big)}\, w_{j;\widehat\beta}
\end{equation} 
which have the properties:
\begin{enumerate}
\item $W_0$ is an eigenfunction of the (critical) affine Toda equation:
\begin{equation}\label{wzereq}
 {\widetilde H} W_0(\widehat\phi) ={\widetilde E}_\lambda \, W_0(\widehat\phi) ,\end{equation}
with 
\begin{equation} {\widetilde H}= \frac{1}{2} \left( \nabla_\phi | \nabla_\phi\right) +\sum_{j=0}^r \nu_j e^{-(\al_j | \widehat\phi)},
\qquad {\widetilde E}_\lambda= E_\lambda+ \nu_0\frac{\partial F}{\partial \nu_0}.  \end{equation}
\item For higher values of $j$, $W_j$ satisfy
\begin{equation}\label{wjeq}
 \left( {\widetilde H}-{\widetilde E}_\lambda\right) \, W_j(\widehat\phi)=\nu_0\frac{\partial W_{j-1}(\widehat\phi)}{\partial \nu_0} \qquad (j=1,2,...)
\end{equation}
\end{enumerate}
\end{thm}
\begin{proof}
Substituting the form \eqref{ansatz} into the deformed affine Toda equation \eqref{deftod}, 
and writing ${\widehat W}_\Lambda^{\widehat\mu,\widehat\mu'}(\widehat{\phi})=e^{\frac{F}{\epsilon}} {\overline W}(\widehat{\phi})$,
we get:
$$ -\nu_0\frac{\partial F}{\partial \nu_0} {\overline W}(\widehat{\phi})-\epsilon \nu_0\frac{\partial {\overline W}(\widehat{\phi})}{\partial \nu_0}
+{\widetilde H} {\overline W}(\widehat{\phi})=E_\lambda \, {\overline W}(\widehat{\phi})$$
Let us collect terms in this identity order by order in $\epsilon$. This gives:
\begin{eqnarray*}
\epsilon^0:&& -\nu_0\frac{\partial F}{\partial \nu_0}\,W_0(\widehat\phi)  +\widetilde{H} \,W_0(\widehat\phi) = E_\lambda \, W_0(\widehat\phi) \\
\epsilon^j:&& -\nu_0\frac{\partial F}{\partial \nu_0} W_j(\widehat\phi)
-\nu_0\frac{\partial W_{j-1}(\widehat\phi)}{\partial \nu_0}+\widetilde{H}\, W_j(\widehat\phi) = E_\lambda \, W_j(\widehat\phi)
\end{eqnarray*} 
which respectively boil down to \eqref{wzereq} and \eqref{wjeq}.

We claim uniqueness of the solutions obtained in this way using the following considerations. Equation \eqref{wzereq} fixes 
$W_0(\widehat\phi)$ up to an overall constant $C$, which may possibly a power series of $\bnu^\delta$, but is
independent of $\phi$. It also fixes the eigenvalue ${\widetilde E}_\lambda=E_\lambda+\nu_0\frac{\partial F}{\partial \nu_0}$ as a power series
of $\bnu^\delta$ equal to $E_\lambda$ at $\bnu=0$. This function is our main interest, although we may also use $W_0(\widehat\phi)$ to solve, using
\eqref{wjeq}, for the functions $W_j(\widehat\phi)$ as power series. The solution of (\ref{wzereq}-\ref{wjeq}) as power series (\ref{formF}-\ref{formWj})
is therefore unique up to the overall multiplicative constant $C$, and the Theorem follows.
\end{proof}

The essential singularity arises from the singularity of the Boltzmann weights at points in the root lattice proportional to $\delta$, when $\epsilon\to 0$ (see Remark \ref{vaffrem}). The factor $e^{\frac{F}{\epsilon}}$ is a result of a 
re-summation of contributions 
from the paths that cross this ``diagonal" $\{\widehat\beta =\beta_0 \delta: \beta_0\in\N\}$, where they pick up a divergent contribution
of the form $\frac{1}{\beta_0\epsilon}$.

Let us now write the recursion relations resulting from (\ref{wzereq}-\ref{wjeq}) explicitly.
Substituting (\ref{formF}-\ref{formWj}) into the equations \eqref{wzereq} and \eqref{wjeq}, we get the following recursion relations:
\begin{equation}\label{recucrit}
v(\beta-\beta_0\theta)w_{j;\widehat\beta} +  \sum_{m\geq 1} m a_{m}w_{j;\widehat\beta-m\delta} =-\beta_0 w_{j-1;\widehat\beta}
+\sum_{i=0}^rw_{j;\widehat\beta-\al_i}, \quad j\geq 0
\end{equation}
using the conventions that $w_{-1;\widehat\beta}=0$, $w_{j;\widehat\gamma}=0$ if 
$\widehat\gamma\not\in \widehat Q_+$, 
and $w_{0,0}=1$.

As explained above, this recursion determines the coefficients $a_m$, $m\in \Z_{>0}$ entirely as rational functions of the $\lambda_i$'s,
as well as all $w_{j;\widehat\beta}$. However 
$w_{0,\widehat\beta}$ is only determined up to an overall multiplicative constant $C$, 
which may be
an arbitrary series expansion of the form $C=\sum_{m\in\Z_+} {\widehat\bnu}^{m\delta} c_m $. This constant is fixed by taking the
limit $\epsilon\to 0$ of the non-critical Whittaker function. 
We may normalize the function by requiring that 
all diagonal coefficients in $W_0$ be zero, except for $w_{0;0}=1$, namely $w_{0;\beta_0\delta}=0$ for all $\beta_0>0$. 

This fixes uniquely the solution of the recursion relations \eqref{recucrit}. With this choice,
when $\widehat\beta=m\delta$ and $j=0$, we get
$$ a_m=\frac{1}{m} \sum_{i=0}^r w_{0;m\delta-\al_i},$$
which determines the function $F$ in terms of the near-diagonal coefficients $w_{0;m\delta-\al_i}$. Subsequently,
when $j=0$ and $\widehat\beta\neq \beta_0 \delta$ we have a recursion
$$ w_{0;\widehat\beta} =\frac{1}{v(\beta-\beta_0\theta)}\left\{\sum_{i=0}^rw_{0;\widehat\beta-\al_i} -\sum_{m\geq 1} \sum_{i=0}^r w_{0;m\delta-\al_i}w_{0;\widehat\beta-m\delta} \right\}. $$

\begin{remark}
Note that this recursion relation is no longer that of a path partition function with local weights, as the coefficient $w_{0;\widehat\beta}$
is expressed in terms of all the $w_{0;\widehat\beta-m\delta}$, $m\geq 1$. This may simply indicate that the path vectors 
$|\bp\rangle$ are no longer the ``right" spanning set in this case.
\end{remark}

\subsection{Example: $A_1^{(1)}$ at critical level $k=-2$}
Let us consider the case of $\widehat G=\widehat A_1$. Writing:
$$ F=\sum_{m\geq 0} (\nu_0\nu_1)^m a_m, \qquad W_j(\widehat\phi)=\sum_{\beta_0,\beta_1\geq 0} \nu_0^{\beta_0}\nu_1^{\beta_1} 
e^{(\lambda_1+1-2(\beta_1-\beta_0))(\phi_1-\phi_0)}  w_{j;\beta_0,\beta_1} $$
the recursion relation \eqref{recucrit} becomes:
$$ (\beta_1-\beta_0)(\lambda_1+1-\beta_1+\beta_0)w_{j;\beta_0,\beta_1} 
+\sum_{m=1}^{{\rm Min}(\beta_0,\beta_1)} m a_m w_{j;\beta_0-m,\beta_1-m}
=w_{j;\beta_0-1,\beta_1}+w_{j;\beta_0,\beta_1-1} - \beta_0 w_{j-1;\beta_0,\beta_1}.$$

When $j=0$, this gives
$$(\beta_1-\beta_0)(\lambda_1+1-\beta_1+\beta_0)w_{0;\beta_0,\beta_1}=w_{0;\beta_0-1,\beta_1}+w_{0;\beta_0,\beta_1-1} 
-\sum_{m=1}^{{\rm Min}(\beta_0,\beta_1)} ma_m w_{0;\beta_0-m,\beta_1-m}$$
The solution with $w_{0;m\delta}=0=w_{0;m,m}$ for all $m>0$ is such that
\begin{equation}\label{eqf} a_m=\frac{1}{m}\left(w_{0;m-1,m}+w_{0;m,m-1}\right) \qquad (m\geq 1).\end{equation}
and for all $n\neq p \in \Z_+$:
\begin{equation}\label{eqw} w_{0;n,p}=\frac{1}{(p-n)(\lambda_1+1-p+n)}\left\{w_{0;n-1,p}+w_{0;n,p-1} 
-\sum_{m=1}^{{\rm Min}(n,p)} ma_m w_{0;n-m,p-m}\right\}
\end{equation}
Substituting \eqref{eqf} into \eqref{eqw} yields the following recursive definition of the $w$'s for $n\neq p$:
\begin{eqnarray*}
w_{0;n,p}&=&\frac{1}{(p-n)(\lambda_1+1-p+n)}\times \\
&&\quad \times\left\{w_{0;n-1,p}+w_{0;n,p-1} 
-\sum_{m=1}^{{\rm Min}(n,p)} \left(w_{0;m-1,m}+w_{0;m,m-1}\right) w_{0;n-m,p-m}\right\}
\end{eqnarray*}
while $w_{0;n,n}=\delta_{n,0}$.
For instance, we find:
\begin{eqnarray*} w_{0;1,0}&=&-\frac{1}{\lambda_1+2}, \quad w_{0;0,1}=\frac{1}{\lambda_1} , \\
w_{0;2,0}&=&\frac{1}{2(\lambda_1+1)(\lambda_1+2)} ,\quad w_{0;1,1}=0,\quad w_{0;0,2}=\frac{1}{2(\lambda_1-1)(\lambda_1-2)}  ,\\
w_{0;3,0}&=&-\frac{1}{6(\lambda_1+1)(\lambda_1+2)(\lambda_1+3)} ,\quad w_{0;2,1}=-\frac{12 + 6\lambda_1+\lambda_1^2}{
 2 \lambda_1( \lambda_1+2)^3 (\lambda_1+3 )} ,\\
w_{0;1,2}&=&-\frac{4-2\lambda_1+\lambda_1^2}{2 (\lambda_1-1)\lambda_1^3 (\lambda_1+2 )} ,\quad w_{0;0,3}= 
\frac{1}{6(\lambda_1-1)(\lambda_1-2)(\lambda_1-3)}\ .
\end{eqnarray*}
We also get:
\begin{eqnarray*}a_1&=&\frac{2}{\lambda_1(\lambda_1+2)},\quad a_2=\frac{(12 + 10\lambda_1+ 5\lambda_1^2)}{\lambda_1^3 (\lambda_1-1)(\lambda_1+2)^3 (\lambda_1+3)}\\
a_3&=&\frac{16 (96 + 152 \lambda_1+ 112 \lambda_1^2 + 36 \lambda_1^3 + 9 \lambda_1^4)}{3\lambda_1^5(\lambda_1-1)(\lambda_1-2) 
(\lambda_1+2)^5 (\lambda_1+3) (\lambda_1+4)} \end{eqnarray*}

\section{Whittaker vectors and functions for $U_q(A_r)$}
Our path model solution for Whittaker vectors and functions can be extended in a straightforward manner to the case of the quantum algebra $U_q(\mathfrak{sl}_{r+1})$ for generic $q$. The combinatorial data, such as the paths, is the same, but there is a key difference in that the weights defined on the path are slightly less local: They depend both on the vertices visited by the paths and on the edges.

Let us supply the necessary definitions in this case. We have the algebra $U_q(A_r)$ with generators $\{E_i, F_i, K_i^{\pm 1}\}$ and relations
\begin{equation}
[E_i,F_j] = \delta_{i,j}\, \frac{K_i-K_i^{-1}}{q-q^{-1}} ,
\quad K_i E_j = q^{C_{j,i}} E_j K_i, \quad K_i F_j = q^{-C_{j,i}} F_j K_i \quad (i,j\in [1,r]) ,
\end{equation}
together with the Serre relations
\begin{eqnarray*}
\begin{array}{c}
 E_i^2 E_j - (q+q^{-1})E_i E_j E_i + E_j E_i^2 = 0,\\
 F_i^2 F_j - (q+q^{-1})F_i F_j F_i + F_j F_i^2 = 0, 
 \end{array} \qquad j=i\pm1.
\end{eqnarray*}

The combinatorial data of roots, weights and the Cartan matrix 
are the same as for finite-type $\g=A_r$. Combinatorially, the irreducible Verma modules are the same as in the non $q$-deformed case. We use the same notation as in the case of finite Lie algebras, with
irreducible highest weight Verma modules $V_{\lambda}=U\big(\{F_i\}_{i\in [1,r]}\big) \, |\lambda\rangle$ 
with highest weight vector $|\lambda \rangle$ and generic weight $\lambda$,
such that $E_i\, |\lambda\rangle=0$  and $K_i\, |\lambda\rangle= q^{\lambda_i}\, |\lambda\rangle$ for all $i\in [1,r]$. 
Again, we have the restricted dual module $V_\lambda^*=\langle \lambda| \, U\big(\{E_i\}_{i\in [1,r]}\big)$ with left highest weight
vector $\langle \lambda|$ such that $\langle \lambda| F_i =0$,  
$\langle \lambda|K_i=q^{\lambda_i}\,  \langle \lambda\vert$ for all $i$ and $\langle \lambda|\lambda\rangle =1$.

\subsection{The q-Whittaker vector: Definition}

\begin{defn}
The Whittaker vector for a highest weight module $V_{\lambda}$ and nilpotent character $\mu=(\mu_1,...,\mu_r)$
is defined to be the unique element $|w\rangle$ in the completion of $V_{\lambda}$ such that $\langle \lambda |w\rangle =1$ and:
\begin{equation}\label{defcon}
E_i\, |w\rangle =\mu_i \, (K_i)^{i-1} |w\rangle \qquad (i\in [1,r]).
\end{equation}
\end{defn}

\begin{remark} The factors $(K_i)^{i-1}$ appear due to reasons of compatibility with the Serre relations, which should act by 0 on the Whittaker vector.
One may use more general defining relations than \eqref{defcon}, 
by picking any set of integers $m_1,...,m_r$ on the nodes of the Dynkin diagram of $A_r$ with the property
 $|m_{i+1}-m_i|=1$ and replacing \eqref{defcon} with $E_i\, |w\rangle =\mu_i \, (K_i)^{m_i} |w\rangle$ for all $i$.
Here we chose to use $m_i=i-1$ for all $i$, as in e.g. \cite{Feigin}.
Compatibility with the Serre relations is a consequence of the $q$-binomial identity
$1-q^\epsilon (q+q^{-1})+q^{2\epsilon} =0$, where $\epsilon=\pm 1$, since
\begin{eqnarray*} &&\left(E_iE_{i+1}^2 -(q+q^{-1})E_{i+1}E_iE_{i+1}+E_{i+1}^2E_i\right) |w\rangle\\
&&\quad =\left(1-(q+q^{-1}) q^{m_i-m_{i+1}}+ q^{2(m_i-m_{i+1})}\right)\mu_i\mu_{i+1}^2\, K_i^{m_i}K_{i+1}^{2m_{i+1}}|w\rangle=0.
\end{eqnarray*}
\end{remark}

We look for an expression for $|w\rangle$ as a linear combination 
of non-independent vectors indexed by the same paths on the root lattice as in the classical case. The bijection between paths and vectors in $V_\lambda$ is adjusted, as we use the generators 
\begin{equation}\label{newF} {\overline{F}}_i= F_i \, (K_i)^{i-1}, \end{equation}
which are better adapted to our definition of the Whittaker vector.

\begin{defn}\label{qmapvec}
Let  ${\mathcal P}_\beta$ be the set of paths $\bp=(p_0,...,p_N)$, with $p_0=0, p_N=\beta\in Q_+$,  and $p_k-p_{k-1}=\al_{i_k}$ for $k=1,...,N$. We define a map from $\mathcal P_\beta\to V_\lambda$ by:
\begin{equation}\label{qp}
|\bp\rangle ={\of}_{i_N}{\of}_{i_{N-1}}\cdots {\of}_{i_1}\, |v\rangle , \quad \bp\in \mathcal P_\beta.
\end{equation}
\end{defn}

As in the classical case, we look for the vector $|w\rangle$ of the form:
\begin{equation}\label{qwit}
|w\rangle=\sum_{\beta\in Q_+} {\boldsymbol \mu}^\beta \sum_{\bp \in {\mathcal P}_\beta} y(\bp)|\bp \rangle
\end{equation}
with $|\bp\rangle$ as in \eqref{qp} and some coefficients $y(\bp)$ which satisfy an appropriate recursion relation. 

Writing the defining condition \eqref{defcon} for $|w\rangle$ in the form 
\begin{equation}\label{newqwit}
{\oE}_i|w\rangle =\mu_i|w\rangle\end{equation} 
where 
\begin{equation}\label{newE}{\oE}_i=K_i^{-(i-1)}E_i\ ,
\end{equation} 
we have the recursion relation for the coefficients:
\begin{lemma}
A sufficient condition for $|w\rangle$ of \eqref{qwit} to satisfy the Whittaker vector conditions \eqref{defcon} 
is that  the coefficients $y(\bp)$ obey the following system:
\begin{equation} \label{overdet}
y(\bp)=\sum_{k=0}^N q^{\sum_{\ell=k+1}^N (i-{i_\ell})C_{i_\ell,i}}
\big[(\lambda-p_k\vert \al_i)\big]\, y(\bp_{k,i}) \qquad ( i\in [1,r])
\end{equation}
where we have used the notation $[x]$ for the $q$-number 
\begin{equation}\label{qnumber} [x]=\frac{q^x-q^{-x}}{q-q^{-1}}, \end{equation}
and $\bp_{k,i}$ for the ``augmented path" \eqref{augmented}. 
\end{lemma}
\begin{proof}
Use \eqref{newqwit}. We first compute ${\oE}_i|w\rangle$ by use of the commutation relation
$$ {\oE}_i \,{\of}_j=\delta_{i,j}\frac{K_i-K_i^{-1}}{q-q^{-1}} +q^{(i-j)C_{j,i}} {\of}_j\, {\oE}_i$$
Since ${\oE}_i \,|\bp\rangle=0$ if the path $\bp$ has no step $\al_i$, we may write
\begin{eqnarray*}
{\oE}_i \,|w\rangle&=& \sum_{\beta\in \al_i+Q_+}\bmu^{\beta-\al_i} \mu_i 
\sum_{\bp\in {\mathcal P}_\beta\atop p=(p_0,...,p_{N+1})}
y(\bp)\sum_{k=0}^N q^{\sum_{\ell=k+1}^N (i-{i_\ell})C_{i_\ell,i}}\delta_{p_{k+1}-p_k,\al_i}\, \delta_{\bp'_{k,i},\bp}\, \big[(\lambda-p_k|\al_i)\big]\, 
|\bp'\rangle\\
&=& \mu_i \sum_{\beta'\in Q_+}\bmu^{\beta'}\sum_{\bp'\in {\mathcal P}_{\beta'}\atop \bp'=(p_0',...,p_N')} \sum_{k=0}^N 
 q^{\sum_{\ell=k+1}^N (i-{i_\ell})C_{i_\ell,i}} y(\bp'_{k,i})\, \big[ (\lambda-p_k|\al_i)\big] \, |\bp'\rangle\\
\mu_i |w\rangle&=&\mu_i\sum_{\beta\in Q_+}\bmu^{\beta}\sum_{\bp\in {\mathcal P}_\beta} y(\bp)\, |\bp\rangle 
\end{eqnarray*}
and the Lemma follows by identifying the coefficients of $|\bp\rangle$ in the last two lines.
\end{proof}


We will give below a solution $y(\bp)=x_q(\bp)$ of the system \eqref{overdet}, 
which has the property that $x_q(\bp)$ is a product of local weights over the path $\bp$.

A path model is an assignment of weights to a set of paths. In this case we need to define 
edge weights, rather than vertex weights as in the classical case. The analogy in the continuum is that the weight depends not just on the local position on the path but also on the first derivative.
\begin{defn}\label{edgp}
An edge-weight path model on $Q_+$ is a map which assigns a weight to each edge of the path. Let $\bp$ be a path as before, with $p_{i}-p_{i-1}\in \Pi$ for all $i$. 
An edge weight $v(p_{i-1},p_i)$ is a function of $p_i$ and of $p_i-p_{i-1}$.
The weight of a path is the product of the edges traversed by the paths, $y(\bp)$:
$$ y(\bp)=\prod_{i=1}^N \frac{1}{v(p_{i-1},p_i)} $$
and the partition function is defined as before by
$$ Z_{\beta}=\sum_{\bp\in {\mathcal P}_\beta} y(\bp) $$
\end{defn}

\begin{lemma}\label{qrec}
The partition function for an edge-weight path model satisfies the following recursion relation:
\begin{equation}
Z_{\beta}=\sum_{i=1}^r \frac{1}{v(\beta-\al_i,\beta)} Z_{\beta-\al_i}
\end{equation}
\end{lemma}
This is to be compared the to the equation satisfied by the partition function with vertex weights Equation \eqref{recur}.

Given the path model, we can define a vector partition function using the map from $\mathcal P_\beta$ to $V_\lambda$ of Definition \ref{qmapvec},
\begin{equation} \label{parvec}
|Z_{\beta}\rangle=\sum_{\bp\in {\mathcal P}_\beta} y(\bp)\, |\bp\rangle .
\end{equation}
We will use this definition below.

\subsection{Quantum Whittaker vectors}
Local edge weights will be defined using the following function.
\begin{defn}
Given an $\mathfrak{sl}_{r+1}$ weight $\lambda$ and an element $\beta$ of the positive root lattice $Q_+$, let
\begin{equation}\label{qweightN}
v_q(\beta)=\frac{1}{(q-q^{-1})^2}\sum_{i=0}^r q^{2(\lambda+\rho|\omega_i-\omega_{i+1})}(1-q^{2(\beta|\omega_{i+1}-\omega_i)})
\end{equation}
with the convention that $\omega_0=\omega_{r+1}=0$.
\end{defn}

In components we may rewrite \eqref{qweightN} as:
\begin{equation}{\bar v}_q(\beta_1,...,\beta_r)
=\frac{1}{(q-q^{-1})^2}\sum_{i=0}^r q^{2(\gamma_i-\gamma_{i+1})}(1-q^{2(\beta_{i+1}-\beta_i)})
\end{equation}
where 
\begin{equation}\label{gammadef}
\gamma_i=\sum_j (C^{-1})_{i,j}(\lambda_j+1)=\frac{i(r+1-i)}{2}
+\frac{1}{r+1}\left\{(r+1-i)\sum_{j=1}^i j\lambda_j+i\sum_{j=i+1}^r (r+1-j)\lambda_j\right\} ,
\end{equation}
and $\lambda_i=(\lambda|\alpha_i^\vee)$.

\begin{lemma}\label{wdif}
The function $v_q(\beta)$ \eqref{qweightN} satisfies the difference equations:
\begin{equation} 
v_q(\beta+\al_i)-v_q(\beta)=q^{(\lambda+\rho-\beta| \omega_{i-1}-\omega_{i+1})} \,  \Big[(\lambda-\beta|\al_i)\Big]
\qquad (i=1,2...,r)
\end{equation}
\end{lemma}
\begin{proof}
By direct computation:
\begin{eqnarray*}
v_q(\beta+\al_i)-v_q(\beta)&=& \frac{1}{(q-q^{-1})^2}\left( -q^{2(\lambda+\rho-\beta|\omega_i-\omega_{i+1})}(q^{-2}-1)
-q^{2(\lambda+\rho-\beta|\omega_{i-1}-\omega_{i})}(q^2-1)\right)\\
&=& \frac{1}{(q-q^{-1})}\left(q^{-2+2(\lambda+\rho-\beta|\omega_i-\omega_{i+1})}-q^{2+2(\lambda+\rho-\beta|\omega_{i-1}-\omega_{i})}\right)\\
&=&q^{(\lambda+\rho-\beta|\omega_{i-1}-\omega_{i+1})} \,  \Big[(\lambda-\beta|\al_i)\Big]
\end{eqnarray*}
by using $\al_i=2\omega_i-\omega_{i+1}-\omega_{i-1}$.
\end{proof}

This prompts the following:
\begin{defn}\label{taupathdef}
We define the weights
\begin{equation}\label{defwi}
v^{(i)}(\beta):=q^{\tau_i(\beta)}\, v_q(\beta)
\end{equation}
where 
\begin{equation}\label{deftaui}
\tau_i(\beta)=(\lambda+\rho-\beta|\omega_{i+1}-\omega_{i-1})
\end{equation}
with the property:
\begin{equation}\label{precur}
v^{(i)}(\beta+\al_i)-v^{(i)}(\beta)= \Big[(\lambda-\beta|\al_i)\Big]
\end{equation}
\end{defn}
In components, we have:
$$\tau_i(\beta)=\gamma_{i+1}-\gamma_{i-1}+\beta_{i-1}-\beta_{i+1}$$
with $\gamma_i$ as in \eqref{gammadef}.


In the following we will consider the quantum path model for $A_r$ defined as follows.
\begin{defn}\label{qpathdef}
The quantum path model for $A_r$ is the edge-weight path model on $Q_+$ of Definition \ref{edgp}, with the edge weights:
\begin{equation}\label{qw}
v(p-\al_i,p)=v^{(i)}(p) \qquad (p\in Q_+^*) 
\end{equation}
where $v^{(i)}$ are as in (\ref{defwi}-\ref{deftaui}). Each path $\bp=(p_0,...,p_N)$ on $Q_+$ receives a weight
\begin{equation}\label{qx}
x_q(\bp):=\prod_{k=1}^N \frac{1}{v^{(i_k)}(p_k)}
\end{equation}
where $p_k-p_{k-1}=\al_{i_k}$ for $k=1,2,...,N$. The corresponding partition function $Z_\beta$ satisfies the recursion relation:
\begin{equation}\label{qrecuZ} v_q(\beta)\, Z_\beta=\sum_{i=1}^r q^{-\tau_i(\beta)}\, Z_{\beta-\al_i}\end{equation}
\end{defn}

The above definitions guarantee the following cancellation lemma.

\begin{lemma}\label{cancelem}
The weights $v^{(i)}$ above satisfy the following general cancellation identity:
\begin{equation}
\frac{v^{(i)}(\beta+\al_i)}{v^{(j)}(\beta+\al_i)} \, 
\frac{v^{(j)}(\beta)}{v^{(i)}(\beta)}=q^{(i-j)C_{j,i}}
\end{equation}
for any $\beta\in Q_+$, and any $i,j\in [1,r]$.
\end{lemma}
\begin{proof}
Due to the form of the weights $v^{(i)}$ the statement of the lemma reduces to the identity
\begin{equation}
\tau_i(\beta+\al_i)-\tau_i(\beta)+\tau_j(\beta)-\tau_j(\beta+\al_i)=(i-j)C_{j,i}
\end{equation}
Noting that
\begin{equation}\label{usefultau}\tau_i(\beta)=\tau_i(\beta+\al_i)
\end{equation}
as $\tau_i(\beta)$ is independent of $\beta_i$, we easily compute:
\begin{equation*}
\tau_i(\beta+\al_i)-\tau_i(\beta)+\tau_j(\beta)-\tau_j(\beta+\al_i)=\tau_j(\beta)-\tau_j(\beta+\al_i)
=(\al_i|\omega_{j+1}-\omega_{j-1})=(i-j)C_{j,i}
\end{equation*}
\end{proof}

We are now ready to state the main result of this section.
\begin{thm}\label{construcq}
The Whittaker vector $|w\rangle$ of $U_q(A_r)$ 
for the h.w.r. $V_{\lambda,v}$, characters $(K_i)^{i-1}$ and eigenvalues $\mu=(\mu_1,...,\mu_r)$
is the generating series for the partition functions of the quantum path model for $A_r$ of Definition \ref{qpathdef}, namely:
\begin{equation}\label{ansa}
|w\rangle = \sum_{\beta\in Q_+}{\boldsymbol \mu}^\beta \,  |Z_{\beta}\rangle 
\end{equation}
with $|Z_{\beta}\rangle$ as in \eqref{parvec}, with the weights 
$y(\bp)= x_q(\bp)$ of \eqref{qx}.
\end{thm}
\begin{proof}
We must check that the weights $x_q(\bp)=\prod_{i=1}^N \frac{1}{v^{(i_j)}(p_j)}$, $N$ the length of $\bp$, obey the system \eqref{overdet}.
For $\bp$ of length $N$, $\bp_{k,i}$ has length $N+1$ and weight:
$$x_q(\bp_{k,i}):=\left(\prod_{\ell=1}^k \frac{1}{v^{(i_\ell)}(p_\ell)}\right)\, \frac{1}{v^{(i)}(p_k+\al_i)}\, 
\left(\prod_{\ell=k+1}^N\frac{1}{v^{(i_\ell)}(p_\ell+\al_i)}\right)$$
By Lemma \ref{wdif} and the subsequent Definition \ref{taupathdef}, we may rewrite for any fixed path 
$\bp=(p_0=0,p_1,p_2,...,p_N=\beta)$:
\begin{eqnarray*}
&&\!\!\!\!\!\!\!\!\!\!\! \sum_{k=0}^N q^{\sum_{\ell=k+1}^N (i-i_\ell)C_{i_\ell,i}}\big[(\lambda-p_k | \al_i)\big]\, x_q(\bp_{k,i})\\
&&=
\sum_{k=0}^N  q^{\sum_{\ell=k+1}^N (i-i_\ell)C_{i_\ell,i}}(v^{(i)}(p_k+\al_i)-v^{(i)}(p_k)) \, x_q(\bp_{k,i})\\
&&=-\underbrace{v^{(i)}(p_0)q^{\sum_{\ell=1}^N (i-i_\ell)C_{i_\ell,i}}\prod_{j=1}^N \frac{1}{v^{(i_j)}(p_j+\al_i)}}_\text{$=0$}\\
&&
+\sum_{k=0}^N q^{\sum_{\ell=k+2}^N (i-i_\ell)C_{i_\ell,i}}\prod_{m=1}^{k}\frac{1}{v^{(i_m)}(p_m)} 
\prod_{m=k+2}^N\frac{1}{v^{(i_m)}(p_m+\al_i)} \times \\
 &&  \times \underbrace{\left(q^{(i-i_{k+1})C_{i_{k+1},i}}\frac{v^{(i)}(p_k+\al_i)}{v^{(i)}(p_k+\al_i)v^{(i_{k+1})}(p_{k+1}+\al_i)} -
 \frac{v^{(i)}(p_{k+1})}{v^{(i_{k+1})}(p_{k+1})v^{(i)}(p_{k+1}+\al_i)} \right)}_\text{$=0$ by Lemma \ref{cancelem}}\\
 &&+ \left(\prod_{m=1}^{N} \frac{1}{v^{(i_m)}(p_m)}\right)\, \frac{v^{(i)}(p_N+\epsilon_i)}{v^{(i)}(p_N+\epsilon_i)}\\ 
&&=x_q(\bp)
\end{eqnarray*}
which is nothing but the equation \eqref{overdet}.
In the above, we have used $v^{(i)}(p_0)=0$ as $p_0=0$ and
the cancellation Lemma \ref{cancelem} for $\beta=p_{k+1}$ and $j=i_{k+1}$ to eliminate all terms but the last one.
We may picture the cancellation above as occurring between the two paths $\bp_{k,i}$ and $\bp_{k+1,i}$ 
that differ only by the order in which the two steps $\al_i,\al_{j}$, $j=i_{k+1}$,  are taken from $p_k$:
$$\raise -.5cm \hbox{\epsfxsize=4.cm \epsfbox{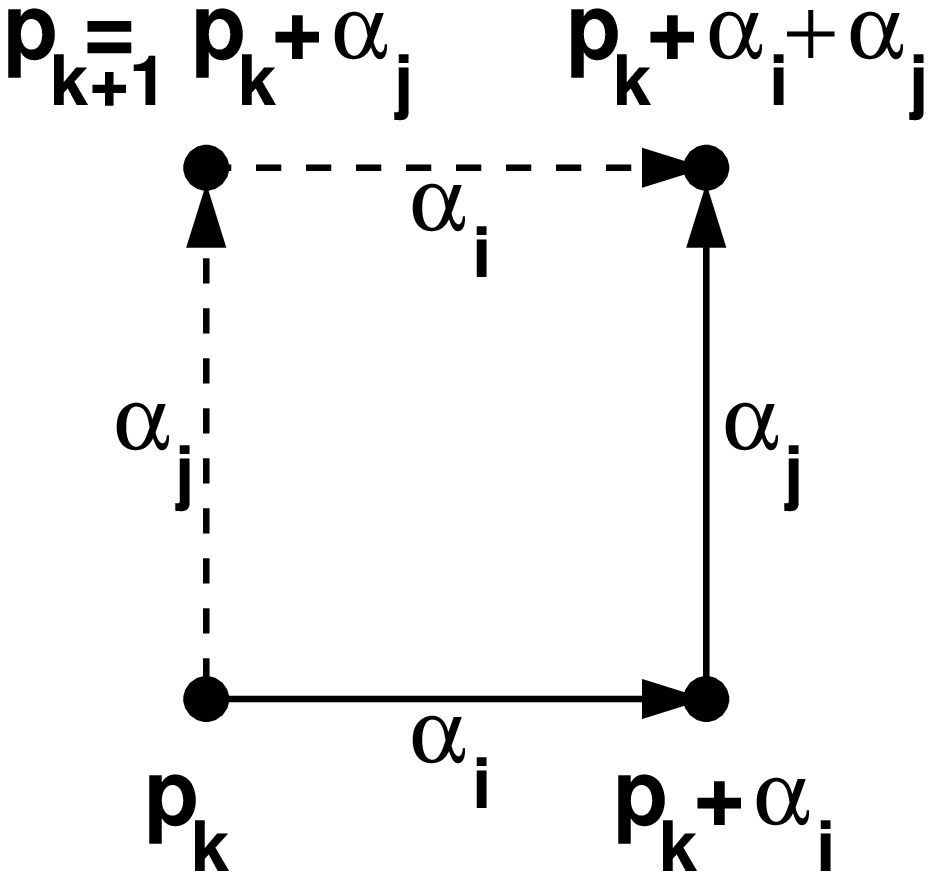} }  $$
\end{proof}

\begin{example}
For $G=A_1$, we have $C=(2)$, $\gamma_1=\frac{\lambda_1+1}{2}$ and:
\begin{equation}v^{(1)}(\beta)=v_q(\beta)=\frac{q^{-\lambda_1-1}(1-q^{2\beta})+q^{\lambda_1+1}(1-q^{-2\beta})}{(q-q^{-1})^2}
=[\beta][\lambda_1-\beta+1] 
\end{equation}
and the q-Whittaker vector reads:
\begin{equation}
|w\rangle =\sum_{\beta\in \Z_+} \frac{\mu^\beta}{\prod_{j=1}^\beta [j][\lambda-j+1]} F^\beta \, |v\rangle 
\end{equation}
as ${\of}=F$.
\end{example}

\begin{example}\label{sl3weights}
For $G=A_2$,  we have $C=\begin{pmatrix}2 & -1\\ -1 & 2\end{pmatrix}$, 
$\gamma_1=1+\frac{2\lambda_1+\lambda_2}{3}$, $\gamma_2=1+\frac{\lambda_1+2\lambda_2}{3}$ and:
\begin{eqnarray*}v_q(\beta_1,\beta_2)&=&
\frac{q^{-2\gamma_1}(1-q^{2\beta_1})+q^{2(\gamma_1-\gamma_2)}(1-q^{2(\beta_2-\beta_1)})
+q^{2\gamma_2}(1-q^{-2\beta_2})}{(q-q^{-1})^2}\\
&=&[\beta_2][\lambda_1 +\lambda_2+2-\beta_1] q^{\beta_1-\beta_2-\gamma_1+\gamma_2} 
+[\beta_1-\beta_2][\lambda_1+1-\beta_1+\beta_2] q^{-\gamma_2}
\end{eqnarray*}
while
\begin{equation*}v^{(1)}(\beta_1,\beta_2)=q^{\gamma_2-\beta_2} \, v_q(\beta_1,\beta_2)\ ,\qquad 
v^{(2)}(\beta_1,\beta_2)=q^{\beta_1-\gamma_1} \, v_q(\beta_1,\beta_2)
\end{equation*}
\end{example}

\subsection{q-Whittaker function and q-difference equations}

Whittaker functions are defined as scalar products of Whittaker vectors with the insertion of a function of the Cartan generators.
The dual Whittaker vector $\langle w|\equiv {}_\lambda^{\mu'}\!\!\langle w\vert$ is determined uniquely by the conditions that 
$\langle w| \in \langle \lambda| U(\{E_i\})$, $\langle \lambda|$ the left highest weight vector such that $\langle \lambda| F_i=0$ and 
$\langle \lambda| K_i=q^{\lambda_i}\,\langle \lambda|$, and $\langle w| F_i=\mu_i'\, \langle w| (K_i)^{-(i-1)}$ for all $i\in [1,r]$, or equivalently
$\langle w| {\of}_i=\mu_i'\, \langle w|$. 
The paths $\langle \bp\vert$ are obtained by acting on the left highest weight vector $\langle \lambda\vert$
by ${\oE_i}$, hence for a path $\bp=(p_0,p_1,...,p_N)$ with steps $p_k-p_{k-1}=\al_{i_k}$, we have:
$$ \langle \bp\vert =\langle \lambda\vert {\oE}_{i_1}{\oE}_{i_2}\cdots  {\oE}_{i_N} $$
To construct $\langle w|$, we use the anti-automorphism of the quantum algebra:
$$ E_i\to F_i,\quad F_i\to E_i,\quad K_i\to K_i^{-1},\quad q\to q^{-1} $$
which also maps ${\of}_i\to {\oE}_i$ and ${\oE}_i\to {\of}_i$.
We find:
\begin{equation}\label{dualvec}
{}_\lambda^{\mu'}\!\!\langle {w}|= 
\sum_{\beta \in Q_+} {\boldsymbol \mu'}^{\beta} \sum_{\bp\in {\mathcal P}_\beta} x_{q^{-1}}(\bp)\, \langle \bp|
\end{equation}
with $x_q(\bp)$ as in \eqref{qx}.

\begin{defn}
Let $\phi=\sum_{i=1}^r\phi_i \al_i$.
The q-Whittaker function associated to the representation $V_{\lambda}$, with right/left eigenvalues 
$\mu=(\mu_1,...,\mu_r)$ and $\mu'=(\mu_1',...,\mu_r')$
is defined by:
\begin{equation}\label{qwitdef}
W_{\lambda}^{\mu,\mu'} (\phi;q)={}_\lambda^{\mu'}\!\!\langle {w}| \prod_{i=1}^r (K_i)^{-\phi_i} |w\rangle_{\lambda}^\mu
\end{equation}
with ${}_\lambda^{\mu'}\!\!\langle {w}|$ as in \eqref{dualvec}.
\end{defn}

We have the following simple expression for the q-Whittaker function.

\begin{thm}\label{qcompW}
\begin{equation}
W_{\lambda}^{\mu,\mu'} (\phi;q)
=\sum_{\beta\in Q_+} {\boldsymbol \mu}^\beta {\boldsymbol \mu'}^\beta q^{-(\lambda-\beta|\phi)}
{\tilde Z}_{\beta}
\end{equation}
where ${\tilde Z}_{\beta}$ denotes the partition function for the edge-weight path model with weights
\begin{equation}\label{neweights}
{\tilde v}^{(i)}(\beta)= q^{-\tau_i(\beta)} v_{q^{-1}}(\beta) 
\end{equation}
\end{thm}
\begin{proof}
For a given path $\bp=(p_0,...,p_N)$ from the origin $p_0=0$ to $p_N=\beta$, let us first compute the scalar product:
$\Sigma_\bp=\langle \bp|w\rangle$.  Writing $p_k-p_{k-1}=\al_{i_k}$ for $k=1,2,...,N$, we have
$$\Sigma_\bp=\langle \lambda| \prod_{\ell=1}^N {\oE}_{i_\ell} |w\rangle =\prod_{\ell=1}^N \mu_{i_\ell}={\boldsymbol \mu}^\beta$$
where we have iteratively used the Whittaker vector defining conditions ${\oE}_i |w\rangle=\mu_i|w\rangle$.
This allows us to write:
\begin{eqnarray*}W_{\lambda}^{\mu,\mu'}(\phi;q)&=&\sum_{\beta\in Q_+}{\boldsymbol \mu'}^\beta
q^{-(\lambda-\beta|\phi)} \sum_{\bp\in {\mathcal P}_\beta} x_{q^{-1}}(\bp)\, \langle \bp| w\rangle_{\lambda}^{\mu}
\\
&=&\sum_{\beta\in Q_+} {\boldsymbol \mu'}^\beta
q^{-(\lambda-\beta|\phi)} \sum_{\bp\in {\mathcal P}_\beta} x_{q^{-1}}(\bp)\, \Sigma_\bp
=\sum_{\beta\in Q_+} {\boldsymbol \mu}^\beta {\boldsymbol \mu'}^\beta
q^{-(\lambda-\beta|\phi)} \sum_{\bp\in {\mathcal P}_\beta} x_{q^{-1}}(\bp)
\end{eqnarray*}
and the theorem follows.
\end{proof}

\begin{cor}\label{barinv}
The partition function $Z_\beta$ for quantum $A_r$ paths is a so-called ``bar-invariant" quantity,
i.e. it is invariant under the transformation $q\to q^{-1}$: ${\tilde Z}_\beta=Z_\beta$.
\end{cor}
\begin{proof}
We use an alternative proof of Theorem \ref{qcompW} that goes as follows. We first compute the scalar product
$\Sigma_\bp'=\langle w|\bp\rangle$ in a similar way. Writing again $\bp=(p_0,...,p_N)$ and
$p_k-p_{k-1}=\al_{i_k}$ for $k=1,2,...,N$, we find:
$$\Sigma_\bp'=\langle w| \prod_{\ell=1}^N {\of}_{i_{N+1-\ell}} |\lambda\rangle =\prod_{\ell=1}^N \mu_{i_\ell}'={\boldsymbol \mu'}^\beta$$
by use of the left Whittaker defining conditions $\langle w| \prod_{\ell=1}^N {\of}_{i}=\mu_i'\langle w|$.
This allows us to write:
\begin{eqnarray*}W_{\lambda}^{\mu,\mu'}(\phi;q)&=&\sum_{\beta\in Q_+}{\boldsymbol \mu}^\beta
q^{-(\lambda-\beta|\phi)} \sum_{\bp\in {\mathcal P}_\beta} x_{q}(\bp)\, {}_\lambda^{\mu'}\langle w| \bp\rangle
=\sum_{\beta\in Q_+} {\boldsymbol \mu}^\beta
q^{-(\lambda-\beta|\phi)} \sum_{\bp\in {\mathcal P}_\beta} x_{q}(\bp)\, \Sigma_\bp'\\
&=&\sum_{\beta\in Q_+} {\boldsymbol \mu}^\beta {\boldsymbol \mu'}^\beta
q^{-(\lambda-\beta|\phi)} \sum_{\bp\in {\mathcal P}_\beta} x_{q}(\bp)
=\sum_{\beta\in Q_+} {\boldsymbol \mu}^\beta {\boldsymbol \mu'}^\beta
q^{-(\lambda-\beta|\phi)} Z_\beta
\end{eqnarray*}
Comparing with the result of Theorem \ref{qcompW}, and identifying the coefficients of ${\boldsymbol \mu}^\beta {\boldsymbol \mu'}^\beta$
in both expressions, we deduce that $Z_\beta={\tilde Z}_\beta$, and the Corollary follows.
\end{proof}

We are now ready for the final result, expressing that the q-Whittaker function obeys a q-Toda difference equation, equivalent to that of
\cite{Etingof}. 

\begin{defn}
We introduce the shift operators $S_i$, $i\in[0,r+1]$  acting on functions $f(\phi)$ as
\begin{equation} S_i\, f (\phi)=f(\phi-\omega_i) \quad (i\in [1,r]), \quad S_0=S_{r+1}=1 \end{equation}
as well as:
\begin{equation} T_i=S_{i+1}S_i^{-1}
(i\in [0,r]) \end{equation}
\end{defn}

\begin{thm}\label{todaqW}
The modified q-Whittaker function 
\begin{equation}\label{modiq}W(\phi):= q^{\sum_i \phi_i} W_{\lambda}^{\mu,\mu'} (\phi;q^{-1})=
\sum_{\beta\in Q_+} {\boldsymbol \mu}^\beta {\boldsymbol \mu'}^\beta q^{(\lambda+\rho-\beta|\phi)}
Z_{\beta}
\end{equation} 
with $Z_\beta$ as in Definition \ref{edgp} with $x(\bp)=x_q(\bp)$,
satisfies the following q-difference Toda equation:
\begin{eqnarray}
H_q W(\phi)&=&E_{q,\lambda} \, W(\phi)\label{maineqto}\\
H_q&=&\sum_{i=0}^r T_i^2 +(q-q^{-1})^2 \sum_{i=1} ^{r}\mu_i\mu_i'\, q^{-(\phi | \al_i)} T_{i-1} T_{i} \label{hamiq}\\
E_{q,\lambda}&=& \sum_{i=0}^r q^{2(\lambda+\rho|\omega_i-\omega_{i+1})}\label{evalq}
\end{eqnarray}
\end{thm}
\begin{proof}
We first note that
$S_i \,q^{(\lambda+\rho-\beta|\phi)}=q^{(\beta-\lambda-\rho|\omega_i)} q^{(\lambda+\rho-\beta|\phi)}$,
hence $S_i$ acts on $W(\phi)$ by inserting $q^{(\beta-\lambda-\rho|\omega_i)}=q^{\beta_i-\gamma_i}$ 
in the expression \eqref{modiq} for $W(\phi)$ as a sum over $\beta$.
Similarly, $T_i$ acts on $W(\phi)$ by inserting 
$q^{(\lambda+\rho-\beta|\omega_{i}-\omega_{i+1})}=q^{\beta_{i+1}-\beta_i-\gamma_{i+1}+\gamma_i}$, and $T_{i-1}T_i$ by inserting 
$q^{-\tau_i(\beta)}=q^{\beta_{i-1}-\beta_{i+1}+\gamma_{i+1}-\gamma_{i-1}}$ 
in the sum \eqref{modiq}. 
As a consequence, using the expression \eqref{qweightN} for $v_q(\beta)$, we deduce that the difference operator 
$$K_q=\sum_{i=0}^r \frac{q^{2(\lambda+\rho|\omega_i-\omega_{i+1})}-T_i^2}{(q-q^{-1})^2}$$ 
acts on $W(\phi)$ by insertion of $v_q(\beta)$ in the sum \eqref{modiq}.
Using the recursion relation \eqref{qrecuZ} satisfied by $Z_\beta$, we finally compute:
\begin{eqnarray*}
K_q \, W(\phi)&=&\sum_{\beta\in Q_+} {\boldsymbol \mu}^\beta {\boldsymbol \mu'}^\beta q^{(\lambda+\rho-\beta|\phi)}\,
v_q(\beta)\, Z_{\beta}\\
&=&\sum_{\beta\in Q_+} {\boldsymbol \mu}^\beta {\boldsymbol \mu'}^\beta q^{(\lambda+\rho-\beta|\phi)}\,
\sum_{k=1}^r q^{-\tau_k(\beta)} \, {Z}_{\beta-\al_k}\\
&=& \sum_{k=1}^r\mu_k\mu_k' \sum_{\beta\in Q_+} 
{\boldsymbol \mu}^\beta {\boldsymbol \mu'}^\beta 
q^{(\lambda+\rho-\beta-\al_k|\phi)}\,
q^{-\tau_k(\beta+\al_k)} \, {Z}_{\beta}\\
&=& \left(\sum_{k=1}^r\mu_k\,\mu_k'\,q^{-(\al_k|\phi)}T_{k-1}T_k \right) \, W(\phi)
\end{eqnarray*}
and the theorem follows.
\end{proof}

\begin{remark}
The q-Toda Hamiltonian $H_q$ is Etingof's q-Toda operator \cite{Etingof}, obtained from the quantum R-matrix of $U_q(\mathfrak{sl}_{r+1})$.
\end{remark}
\subsection{Examples}
\vskip.1in
For the algebra \noindent{$U_q(\mathfrak{sl}_2)$,} the Whittaker function is
\begin{equation}\label{Wsltwoq} W(\phi)=\sum_{\beta\in \Z_+} (\mu\mu')^{a}q^{\phi(\lambda+1-2\beta)}\,
{Z}_{\beta}
\end{equation}
where 
$$ {Z}_{\beta} =\frac{1}{\prod_{j=1}^\beta [j][\lambda+1-j]} $$
Note that the latter is manifestly bar-invariant (see Corollary \ref{barinv}), as all q-numbers are.
Moreover, we have:
$S_0=S_2=1$, $S_1 f(\phi)= f(\phi-\frac{1}{2})$, $T_0^2 f(\phi)=f(\phi-1)$, $T_1^2 f(\phi)=f(\phi+1)$, $T_0T_1=I$, and:
$$H_q=T_0^2 +T_1^2 +(q-q^{-1})^2 \mu \mu'\, q^{-2\phi} ,\qquad E_{q,\lambda}= q^{\lambda+1}+q^{-\lambda-1}$$
so that, acting on functions $f$ of the variable $\phi$, we have:
$$ H_q \, f(\phi)= f(\phi+1)+f(\phi-1) +(q-q^{-1})^2 \mu \mu'\, q^{-2 \phi} f(\phi)$$
and can we check directly that 
$$ H_q \, W(\phi)= E_{q,\lambda}\, W(\phi) $$
as a consequence of the recursion relation $v_q(\beta) {\hat Z}_\beta={\hat Z}_{\beta-1}$, with 
$$v_q(\beta)=\frac{q^{\lambda+1}(1-q^{-2\beta})+q^{-\lambda-1}(1-q^{2\beta})}{(q-q^{-1})^2}.$$

\vskip.1in

{For the algebra $U_q(\mathfrak{sl}_3)$,} the Whitaker function is
$$W(\phi)=
\sum_{\beta_1,\beta_2\in \Z_+} (\mu_1\mu_1')^{\beta_1} (\mu_2\mu_2')^{\beta_2}q^{\phi_1(\lambda_1+1-2\beta_1)+\phi_2(\lambda_2+1-2\beta_2)}\,
{Z}_{\beta_1,\beta_2},$$
where ${Z}_{\beta_1,\beta_2}$ satisfies the recursion relation (see Example \ref{sl3weights}):
$$v_q(\beta_1,\beta_2)\, {Z}_{\beta_1,\beta_2}=q^{\gamma_2-\beta_2} {Z}_{\beta_1-1,\beta_2}+q^{\beta_1-\gamma_1}{Z}_{\beta_1,\beta_2-1},$$ 
which leads to the solution:
$${Z}_{\beta_1,\beta_2}=\frac{
\prod_{j=1}^{\beta_1 + \beta_2} [\lambda_1 + \lambda_2+2 - j]}{\left(\prod_{j=1}^{\beta_1}
  [j] [\lambda_1+1 - j][\lambda_1 + \lambda_2+2 - j]\right)\left(\prod_{j=1}^{\beta_2}
 [j] [\lambda_2+1 - j][\lambda_1 + \lambda_2+2 - j]\right)}.$$
This is the quantum version of the factorized Bump formula \eqref{bump}. 
Note that this expression is manifestly bar-invariant (see Corollary \ref{barinv}).
Unfortunately, not such nice factorized formula seem to exist for general $A_r$, $r\geq 3$.

The Hamiltonian $H_q$ and eigenvalue $E_{q,\lambda}$ are:
\begin{eqnarray*}H_q&=&T_0^2 +T_1^2 +T_2^2+(q-q^{-1})^2 (\mu_1 \mu_1'\, q^{\phi_2-2\phi_1}T_0T_1+\mu_2 \mu_2'\, q^{\phi_1-2\phi_2}T_1T_2) ,\\
E_{q,\lambda}&=& q^{-2\gamma_1}+q^{2(\gamma_1-\gamma_2)}+q^{2\gamma_2}.\end{eqnarray*}
where
\begin{eqnarray*}
T_0 \,f(\phi_1,\phi_2)&=& f(\phi_1-{\scriptstyle \frac{2}{3}},\phi_2-{\scriptstyle \frac{1}{3}}),\\\ 
T_1\, f(\phi_1,\phi_2)&=& f(\phi_1+{\scriptstyle \frac{1}{3}},\phi_2-{\scriptstyle \frac{1}{3}}),\\
\ T_2\,f(\phi_1,\phi_2)&=& f(\phi_1+{\scriptstyle \frac{1}{3}},\phi_2+{\scriptstyle \frac{2}{3}}).
\end{eqnarray*}

\section{Summary and discussion}

\subsection{Whittaker vectors}
In this paper, we have constructed path models based on Chevalley generators to compute the Whittaker vectors in the completion of irreducible, generic highest weight
Verma modules $V_{\lambda}$ for any finite-dimensional simple Lie algebras and affine Lie algebras, as well as the quantum algebra
$U_q(\mathfrak{sl}_{r+1})$. In all cases, we wrote a combinatorial expression for the Whittaker vectors as linear 
combinations of a generating set of $V_\lambda$
obtained by the free action of the nilpotent subalgebra $\mathfrak{n}_-$ on the highest weight vector $|\lambda\rangle$, 
naturally indexed by paths on the positive root lattice,
starting at the origin. We showed that the coefficients (path weights) in these linear combinations could be chosen in a very simple manner.
In the simple and affine Lie cases, the path weight is taken to be a product of local factors $1/v(p_k)$, 
depending only on each vertex $p_k$ visited by the path
(vertex weights). In the quantum case, the path weight is still a product, but local factors $1/v(p_{k-1},p_k)$
depend on {\it pairs} of consecutive vertices visited by the path (edge weights).

In all cases, our choice of path weights was dictated by the requirement of locality. 
It boils down to difference equations of the form
$v(\beta+\al_i)-v(\beta)= (\lambda-\beta|\al_i)$ for simple or affine cases, or $v^{(i)}(\beta+\al_i)-v^{(i)}(\beta)= [(\lambda-\beta|\al_i)]$
for the quantum case. 
For other quantum algebras $U_q(\mathfrak{g})$ the situation is more complex and will be addressed elsewhere.

\subsection{Whittaker functions}
We showed that the path model could be used to determine the corresponding Whittaker
functions.
Indeed, the recursion relation for
the path partition functions immediately yields an eigenfunction condition for the
corresponding Whittaker functions, for the Toda type second order differential (resp. difference) operator attached to the Lie 
(resp. quantum) algebra.

Solutions of the Toda equation are characterized by their asymptotic properties. The Whittaker functions we obtain using the path model are characterized by the 
fact that they have a well-defined power series expansion in the variables $e^{-\sum_j C_{i,j}\phi_j}$. Up to a factor of $e^{(\lambda+\rho|\phi)}$,
the Whittaker function remains bounded as $\phi_i\to\infty$ in the fundamental Weyl chamber $\sum_j C_{i,j}\phi_j\geq 0$. 
Such solutions are referred
to as {\it fundamental} Whittaker functions \cite{Hashizume}. 

In the case of finite dimensional $\g$, for fixed eigenvalue $E_\lambda$ of the 
Toda Hamiltonian, the images
$W_{\sigma(\lambda)}(\phi)$ of our function $W(\phi)$ under Weyl group reflections 
$\sigma: \lambda\to \sigma(\lambda):=\sigma(\lambda+\rho)-\rho$ are known to form a basis of the corresponding 
eigenspace of the Toda operator. 
Other Whittaker functions of interest are the so-called class I Whittaker functions \cite{Hashizume}, computed in principal series (unitary) representations of the real form of the Lie group, which are not highest weight modules. 
Class I Whittaker functions 
remain bounded in the reflected Weyl chambers as well, and this property allows
to decompose them into the fundamental function basis with simple coefficients \cite{Hashizume}.

The corresponding Whittaker vectors themselves cannot be computed using our path model, which is specifically tailored to work 
for highest weight modules (see the expressions of \cite{KLT} for the  $U_q(\mathfrak{sl}(2,\R))$ case).
Class I q-Whittaker functions for $U_q(\mathfrak{gl}_{n+1})$ were derived
in \cite{Lebedev,Lebedev3} for a related version of the q-Toda operator, 
obtained by acting on $H_q$ of \eqref{hamiq} with an automorphism $\tau$ of the algebra generated by $\{T_i\}_{i\in [0,r]}$ and 
$\{q^{-\sum_j C_{i,j}\phi_j}\}_{i\in [1,r]}$
(see \cite{Etingof}). In this quantum case, the regularity property becomes 
much stronger: the class I q-Whittaker functions actually vanish outside of the fundamental Weyl chamber for
integral values of the variables $n_i=\frac{1}{2}\sum_j C_{i,j}\phi_j$. We have checked for the case of $\mathfrak{sl}_2$ that the class I 
q-Whittaker function of \cite{Lebedev,Lebedev3} is a linear combination of the two fundamental q-Whittaker functions obtained by acting
on $W(n=\phi)$ of \eqref{Wsltwoq} and its Weyl group reflection by the automorphism $\tau$. We expect this property to generalize to higher rank.
Related q-difference equations and their generalizations also occur in the determination of graded characters for fusion products of special modules of the quantum algebras, for which the latter play a role analogous to that of class I q-Whittaker functions. These difference equations arise from the integrability of the so-called quantum Q-systems involving non-commutative variables that generalize characters \cite{DFKplus}.
(See also  \cite{Feigin} for similar formulas expanded in the Gelfand-Zeitlin basis, and \cite{Lebedev3} for a connection to McDonald polynomials).
It is our hope that our new formulas using path models may shed some new light on these connections.

\subsection{Integrability}

The open quantum Toda spin chain is known to be integrable in the classical case, namely there exist a set of higher order
commuting differential operators, simultaneously diagonalized by the Whittaker functions. In turn, these higher order eigenfunction equations
are related to higher degree recursion relations for the {\it same} path partition functions. Let us illustrate this in the case of $G=A_2$
(see Example \ref{sl3funex}). From the factorized form of the partition function \eqref{bump}, it is easy to write other recursion relations
for the partition function function $Z_{\beta_1,\beta_2}$. For instance, we have:
\begin{eqnarray}
\beta_1(\lambda_1+1-\beta_1)(\lambda_1+\lambda_2+2-\beta_1)Z_{\beta_1,\beta_2}
&=&(\lambda_1+\lambda_2+2-\beta_1-\beta_2)Z_{\beta_1-1,\beta_2}\label{a2one}\\
\beta_2(\lambda_2+1-\beta_2)(\lambda_1+\lambda_2+2-\beta_2)Z_{\beta_1,\beta_2}
&=&(\lambda_1+\lambda_2+2-\beta_1-\beta_2)Z_{\beta_1,\beta_2-1}\label{a2two}
\end{eqnarray}
These are not independent, as readily seen by taking their sum, equal to the original recursion relation \eqref{eqa2}
up to an overall factor of $(\lambda_1+\lambda_2+2-\beta_1-\beta_2)$. If we form the combination
$(\lambda_2+1-\beta_2)\times$\eqref{a2one}$-(\lambda_1+1-\beta_1)\times$\eqref{a2two}, we obtain:
$$(\beta_1-\beta_2)(\lambda_1+1-\beta_1)(\lambda_2+1-\beta_2)Z_{\beta_1,\beta_2}
=(\lambda_1+1-\beta_1)Z_{\beta_1,\beta_2-1}-(\lambda_2+1-\beta_2)Z_{\beta_1-1,\beta_2}$$
This leads to a third order differential operator eigenfunction equation for the Whittaker function $W(\phi)$ of the form
$H^{(3)}\,  W(\phi)=E^{(3)} W(\phi)$, with
\begin{eqnarray*}H^{(3)}&=&\frac{(D_1-D_2)(2D_1+D_2)(D_1+2D_2)}{9}-\nu_1 e^{-(2 \phi_1 - \phi_2)} 
(D_1 + 2 D_2) +\nu_2 e^{-(2 \phi_2 - \phi_1)} ( 2 D_1+D_2)\\
E^{(3)}&=&(\lambda_1-\lambda_2) E_{\lambda_1,\lambda_2}- \frac{(\lambda_1-\lambda_2)^3}{9}
=\frac{1}{9}(\lambda_1-\lambda_2)(2\lambda_1+\lambda_2+3)(\lambda_1+2\lambda_2+3)
\end{eqnarray*}
where we used the notations of Example \ref{todasl3}. The higher Hamiltonian $H^{(3)}$ is readily seen to commute with 
the Toda Hamiltonian $H$.


In the case of affine Lie algebras, we have been able to investigate the limit of critical level $k\to -h^\vee$, and shown how
the limiting Whittaker functions involve eigenfunctions of the critical affine Toda operator (related to the closed Toda spin chain) 
with a suitably transformed eigenvalue.
The critical limit is important because the affine algebra acquires an infinite-dimensional center, 
isomorphic to the (deformed) classical $W$-algebras, a phenomenon closely related to the integrability
in the form of the existence of higher order commuting differential operators, simultaneously diagonalized by Whittaker 
functions. It would be interesting to investigate the integrable structure of the non-critical case as well, by use of our path model, 
namely by looking for higher order recursion relations satisfied by our path partition functions.




\bibliographystyle{alpha}

\bibliography{refs}

\end{document}